\documentclass[12pt]{article}
\usepackage{graphicx}
\usepackage{amssymb, amsmath, amstext, amsthm}
\usepackage{float}
\usepackage{natbib}
\usepackage{enumerate}
\usepackage{mathtools}
\mathtoolsset{showonlyrefs}
\usepackage{scrextend}

\usepackage{color}    
\usepackage[usenames,dvipsnames]{xcolor} 

\RequirePackage[colorlinks]{hyperref} 
 
\newtheorem{theorem}{Theorem}[section]
\newtheorem{lemma}[theorem]{Lemma}
\newtheorem{proposition}[theorem]{Proposition}

\theoremstyle{definition}

\newtheorem{example}{Example}[section]
\newtheorem{assumption}{Assumption}

\newtheorem{remark}[theorem]{Remark}

\usepackage{color}

\usepackage{mathrsfs} 
\usepackage{hyperref}
\hypersetup{
colorlinks=true,
citecolor=black,
linkcolor=black,
urlcolor=black
}

\newcommand{\N}{\mathbb{N}}
\newcommand{\R}{\mathbb{R}}

\input{tcilatex}

\title{How to avoid the zero-power trap in testing for correlation}

\author{David Preinerstorfer \\[10pt]
\small{ECARES and SBS-EM} \\[2pt] \small{Universit\'e libre de Bruxelles} \\[2pt]
\small{\href{mailto:david.preinerstorfer@ulb.ac.be}{david.preinerstorfer@ulb.ac.be}
}}

\date{August 2018}

\begin{document}
	
\maketitle

\begin{abstract}
In testing for correlation of the errors in regression models the power of tests can be very low for strongly correlated errors. This counterintuitive phenomenon has become known as the ``zero-power trap''. Despite a considerable amount of literature devoted to this problem, mainly focusing on its detection, a convincing solution has not yet been found. In this article we first discuss theoretical results concerning the occurrence of the zero-power trap phenomenon. Then, we suggest and compare three ways to avoid it. Given an initial test that suffers from the zero-power trap, the method we recommend for practice leads to a modified test whose power converges to one as the correlation gets very strong. Furthermore, the modified test has approximately the same power function as the initial test, and thus approximately preserves all of its optimality properties. We also provide some numerical illustrations in the context of testing for network generated correlation.
\end{abstract}

\section{Introduction}\label{sec:intro}

Testing whether the errors in a regression model are uncorrelated is a standard problem in econometrics. For many forms of correlation under the alternative there are well-established tests available. Two prominent examples are the Durbin-Watson test for serial autocorrelation, and the Cliff-Ord test for spatial autocorrelation. Nevertheless, this type of testing problem is not completely solved, not even in the Gaussian case. This is partly due to the fact that tests for correlation, including the well-established tests mentioned before, do not always behave as they ideally should in finite samples: Whereas the size of most tests can be easily controlled, at least under suitable distributional assumptions such as Gaussianity, their power function can attain very small values in regions of the alternative where the correlation is very strong. This, however, does not match with the intuition that strong correlations should be \emph{easily} detectable from the data, i.e., that  the power of a test for correlation should be close to one if the degree of correlation in the errors is very strong. 

That the power function of a test for correlation can drop to zero as the correlation increases was first formally established in \cite{kramer85}, who considered the power function of the Durbin-Watson test in testing for serial autocorrelation. The results in \cite{kramer85} were extended in later work by \cite{zeisel1989}, \cite{KramerZeisel1990} and \cite{lobus2000}. \cite{kleiber2005} obtained similar results for the Durbin-Watson test when the disturbances are fractionally integrated. \cite{kramer2005finite} proved related results for Cliff-Ord-type tests in case the regression errors are spatially autocorrelated. A unifying general theory that neither relies on the specific form of correlation nor on very special structural properties of the tests was developed recently in \cite{mart10} and \cite{PP17}. We refer the interested reader to the latter articles for formal results and a thorough discussion of the literature. 

The major practical value of the just mentioned articles is of a diagnostic nature: they provide conditions which depend on observable quantities only and which let a user decide whether a particular test is subject to the \emph{zero-power trap}, i.e., whether its power function drops to zero as the correlation increases. This is important, because if it turns out that an initial test is subject to this trap, one may want to use another test. However, one is then confronted with the problem of finding a test that avoids the zero-power trap. One complication is as follows: Typically, the initial test was chosen for a reason, i.e., for its ``optimal'' power properties in certain regions of the parameter space (think of a locally best invariant test). In such situations, one would not just like to use \emph{some} other test that avoids the zero-power trap. Much more likely, one would prefer to \emph{slightly modify} the initial test in such a way that its optimality properties are preserved, at least approximately, but such that its modified version does not suffer from the zero-power trap. Compared to the amount of literature that concentrates on deriving diagnostic tools for detecting the zero-power trap, the attention that has been paid to the question \emph{how} one can construct tests which do not suffer from the zero-power trap is much less. Furthermore, it is not clear how to obtain said ``optimality-preserving'' modifications. The main contribution of the present article is to fill this gap. In the following paragraphs we provide an overview of the article's structure together with a more detailed summary of our contributions.

\smallskip

In Section~\ref{sec:fram} we introduce the framework: the model and the testing problem, some notational conventions and an important class of tests. In Section~\ref{sec:ZPT} we formally define the zero-power trap phenomenon, obtain some sufficient conditions for it from results in \cite{PP17}, and then consider in our general framework the question how often, i.e., for ``how many'' design matrices, the zero-power trap actually arises. We answer this question in Propositions~\ref{prop:smallalpha} and~\ref{prop:mart}. The former proposition proves (and generalizes) an observation already made in the discussion section of \cite{kramer85}. The latter proposition is obtained by generalizing an argument in \cite{martellosio2012testing}, who considered the same question in a spatial autoregressive setting. Essentially, these two propositions show (for the tests based on the specific family of test statistics and the corresponding critical values considered) respectively that~(i)~the zero-power trap arises for generic design matrices (i.e., up to a Lebesgue null set of exceptional matrices) for small enough critical values; and~(ii) for any critical value that leads to a size in~$(0, 1)$ there exists an open set of design matrices for which the zero-power trap arises.

In Section~\ref{sec:avoid} we present three ways to avoid the zero-power trap: 
In Section~\ref{sec:ee} we briefly discuss a test for which \cite{PP17} have shown that it does not suffer from the zero-power trap. This test typically does not have very favorable power properties, apart from the fact that it avoids the zero-power trap. We shall mainly use it later as a building block in our construction of ``optimality-preserving'' tests. In Section~\ref{sec:artreg} we discuss tests that incorporate artificial regressors to avoid the zero-power trap. The suggestion of adding artificial regressors to the regression and to use ``optimal'' tests in this expanded model is present already in \cite{kramer85}, who observed numerically that adding the intercept to a regression without intercept helps to avoid the zero-power trap for the Durbin-Watson test. Our theoretical results in Section~\ref{sec:artreg} exploit results in \cite{PP17}, and are related to the methods in \cite{PP13} and \cite{preinerstorfer2017}, who considered the construction of tests with good size and power properties for testing restrictions on the regression coefficient vector. While the tests in Section~\ref{sec:artreg} are ``optimality-preserving'' to some extent (more specifically they often have the same optimality property as initial tests, but within a smaller class of tests), it turns out that this solution to the zero-power trap is not ideal. For example, the power function of these tests does not increase to one as the strength of the correlation increases (which is the case for the approach outlined in Section~\ref{sec:ee}). 

In Section~\ref{subs:approx} we construct optimality-preserving modifications avoiding the zero-power-trap out of an initial test that suffers from the zero-power trap. Our approach overcomes the limitations of the approaches discussed in Sections~\ref{sec:ee} and~\ref{sec:artreg}. In particular, our method leads to tests that have approximately the same power properties as the initial test. Furthermore, their power converges to one as the strength of the correlation increases. The construction is inspired by the power enhancement principle of \cite{fan2015power} in the formulation used in Section~3 of \cite{kock2017power}. The basic idea of this principle is to improve the asymptotic power of an initial test by using another test, a power enhancement component, which has better asymptotic power properties than the initial test in certain regions of the alternative. Since the theory in \cite{fan2015power} and \cite{kock2017power} is asymptotic, and \emph{the present article is concerned exclusively with finite sample properties}, their results do not apply here. Nevertheless, we can adapt the underlying heuristic to our context: given an initial test that suffers from the zero-power trap, but has favorable power properties in other regions of the alternative, we ``combine'' this initial test with the test from Section~\ref{sec:ee} to obtain an ``enhanced'' test. 

In Section~\ref{sec:num} we compare the approaches for avoiding the zero-power trap discussed in Section~\ref{sec:avoid} numerically. We reconsider an example in \cite{kramer2005finite} in which the Cliff-Ord test turns out to suffer from the zero-power trap. Section~\ref{sec:concl} concludes. All proofs are collected in Appendices~\ref{app:intro}-\ref{app:avoid}.

\section{Framework}\label{sec:fram}

In the present section we introduce the model, the testing problem and some notation, and we discuss an important class of tests. Most of the notational conventions and terminology we use are standard, and coincide to a large extent with the ones in \cite{PP17}. We repeat them here for the convenience of the reader.

\subsection{Model and testing problem}\label{Model}

We consider the linear model
\begin{equation}\label{linmod}
\mathbf{y}=X\beta +\mathbf{u},  
\end{equation}%
where~$X\in \mathbb{R}^{n\times k}$ is a non-stochastic matrix of rank~$k$
with~$0< k<n$, and where~$\beta \in \mathbb{R}^{k}$ is the regression coefficient vector. The disturbance vector~$\mathbf{u}$ is
assumed to be Gaussian with mean zero and covariance
matrix~$\sigma ^{2}\Sigma (\rho )$. Here~$\Sigma (.)$ is a \emph{known}
function from~$[0,a)$ to the set of symmetric and positive definite~$n\times
n$ matrices, and~$a$ is a prespecified positive real number.  Without loss of generality we assume throughout that~$\Sigma (0)$ equals the identity matrix~$I_{n}$. The parameters~$\beta \in \R^k$,~$\sigma \in (0, \infty)$ and~$\rho \in \lbrack 0,a)$ are unknown.  

The Gaussianity assumption could be relaxed considerably. It is imposed mainly to avoid technical conditions that do not deliver deeper insights into the problem. For example, we could replace the Gaussianity assumption by the assumption that the distribution of the error vector~$\mathbf{u}$ is elliptically symmetric without changing any of our results. This and other generalizations are discussed in detail in Section~3 of \cite{PP17}.

Denoting the Gaussian probability measure with mean~$X\beta$ and covariance matrix~$\sigma^2 \Sigma(\rho)$ by~$P_{\beta, \sigma, \rho}$, we see that the model \eqref{linmod} induces the parametric family of distributions 
\begin{equation}
\left\{ P_{\beta ,\sigma ,\rho }:\beta \in \mathbb{R}%
^{k},~\sigma \in (0, \infty),~\rho \in \lbrack 0,a)\right\}  \label{family}
\end{equation}%
on the sample space~$\mathbb{R}^{n}$ equipped with its Borel~$\sigma$-algebra. The expectation operator with respect to (w.r.t.)~$P_{\beta ,\sigma ,\rho }$ will be denoted by~$E_{\beta ,\sigma ,\rho }$. Note that the set of probability measures in the previous display is dominated by Lebesgue measure~$\mu_{\R^n}$ on the Borel sets of~$\R^n$, because~$\Sigma(\rho)$ is positive definite for every~$\rho \in [0, a)$ by assumption.

In the family of distributions~\eqref{family} we are interested in the testing problem~$\rho = 0$ against~$\rho > 0$. More precisely, the testing problem is
\begin{equation}\label{testproblem}
H_{0}:\rho =0,~\beta \in \mathbb{R}^{k},~0<\sigma <\infty \quad \text{ against } \quad
H_{1}:\rho >0,~\beta \in \mathbb{R}^{k},~0<\sigma <\infty ,
\end{equation}%
with the implicit understanding that always~$\rho \in \lbrack 0,a)$. In this testing problem the parameter~$\rho$ is the target of inference, and the regression coefficient vector~$\beta$ and the parameter~$\sigma$ are nuisance parameters. 

Two specific examples that received a considerable amount of attention in the econometrics literature and which fit into the above framework are testing for positive serial autocorrelation and testing for spatial autocorrelation, cf.~Examples~2.1 and~2.2 in \cite{PP17} for details and a discussion of related literature. See also Section~\ref{sec:num} below for more information on testing for spatial autocorrelation and related numerical results.

\subsection{Notation, invariance and an important class of tests} \label{ss:notation}

\subsubsection{Notation}\label{sec:notation}  

All matrices we shall consider are real matrices, the transpose of a matrix~$A$ is denoted by~$A^{\prime }$, and the space spanned by the columns of~$A$ is denoted by~$%
\limfunc{span}(A)$. Given a linear subspace~$L$ of~$\mathbb{R}^{n}$, the
symbol~$\Pi _{L}$ denotes the orthogonal projection onto~$L$, and~$L^{\bot }$
denotes the orthogonal complement of~$L$. Given an~$n\times m$ matrix~$Z$ of
rank~$m$ with~$0\leq m<n$, we denote by~$C_{Z}$ a matrix in~$\mathbb{R}%
^{(n-m)\times n}$ such that~$C_{Z}C_{Z}^{\prime }=I_{n-m}$ and~$%
C_{Z}^{\prime }C_{Z}=\Pi _{\limfunc{span}(Z)^{\bot }}$ where~$I_{r}$ denotes
the identity matrix of dimension~$r$. We observe that every matrix
whose rows form an orthonormal basis of~$\limfunc{span}(Z)^{\bot }$
satisfies these two conditions and vice versa. Hence, any two choices for 
$C_{Z}$ are related by premultiplication by an orthogonal matrix. Let~$l$ be
a positive integer. If~$A$ is an~$l\times l$ matrix and~$\lambda \in \mathbb{%
R}$ is an eigenvalue of~$A$ we denote the corresponding eigenspace by~$%
\limfunc{Eig}\left( A,\lambda \right)$. The eigenvalues of a symmetric
matrix~$B\in \mathbb{R}^{l\times l}$ ordered from smallest to largest and
counted with their multiplicities are denoted by~$\lambda _{1}(B),\ldots
,\lambda _{l}(B)$. We shall sometimes denote~$\lambda _{1}(B)$ by~$\lambda_{\min}(B)$, and~$\lambda _{l}(B)$ by~$\lambda_{\max}(B)$. Lebesgue measure on the Borel~$\sigma$-algebra of~$\R^{n \times l}$ shall be denoted by~$\mu _{\R^{n \times l}}$, and~$\text{Pr}$ is used as a generic symbol for a probability measure. The Euclidean norm of a vector is denoted by~$\|.\|$, a symbol that is also used to denote a matrix norm. 

\subsubsection{Invariance, an important class of tests, and size-controlling critical values}\label{sec:tests}

Given a matrix~$Z\in \mathbb{R}^{n\times m}$ with column rank~$m$ and where~$1\leq m<n$, define the group of bijective transformations (the group action being composition of functions)%
\begin{equation*}
G_{Z}:=\left\{ g_{\gamma ,\theta }:\gamma \in \mathbb{R}\backslash \left\{
0\right\} ,\theta \in \mathbb{R}^{m}\right\} ,
\end{equation*}%
where~$g_{\gamma ,\theta }: \R^n \to \R^n$ denotes the function~$y\mapsto \gamma y+Z\theta$.

Under our distributional assumptions (and if additionally all parameters of the model are identifiable) the testing problem in Equation~\eqref{testproblem} is invariant w.r.t.~the group~$G_X$ (cf.~Section~6 in \cite{lehmannromano}). It thus appears reasonable to consider tests that are~$G_X$-invariant, a property shared by most commonly used tests. Recall that a function~$f$ defined on the sample space (e.g., a test or a test statistic) is called invariant w.r.t.~$G_X$ if and only if for every~$y \in \R^n$ and every~$g_{\gamma, \theta} \in G_X$ it holds that~$f(y) = f(g_{\gamma, \theta}(y))$. A subset~$A$ of~$\R^n$ will be called invariant w.r.t.~$G_X$ if the indicator function~$\mathbf{1}_A$ is~$G_X$-invariant.

In addition to being~$G_X$-invariant, most tests for~\eqref{testproblem} used in practice are non-randomized, i.e., they are indicator functions of Borel sets -- their corresponding rejection regions. An important class of such tests is based on rejection regions of the form 
\begin{equation}
\Phi _{B,c }=\Phi _{B,C_{X}, c }=\left\{ y\in \mathbb{R}%
^{n}:T_{B}\left( y\right) > c \right\} ,  \label{quadratic}
\end{equation}
where~$c \in \R$ is a critical value and the test statistic
\begin{equation}\label{T_quadratic}
T_{B}\left( y\right) =T_{B,C_{X}}\left( y\right) =
\begin{cases}
y^{\prime }C_{X}^{\prime }BC_{X}y/\Vert {C_{X}y}\Vert ^{2} & \text{if } y\notin \limfunc{span}(X) \\[4pt] 
\lambda _{1}(B) & \text{if } y\in \limfunc{span}(X).
\end{cases}
\end{equation}
Here~$B\in \mathbb{R}^{(n-k)\times (n-k)}$ is a symmetric matrix,
which typically depends on~$X$ and the function~$\Sigma$. Recall that the matrix~${C_{X}}$
satisfies~${C_{X}C}^{\prime }{_{X}=I}_{n-k}$ and~${C}^{\prime }{%
_{X}C_{X}=\Pi }_{\text{$\limfunc{span}$}(X)^{\bot }}$ (cf.~Section~\ref{sec:notation}). Clearly, the test statistic~$T_{B}$ is~$G_{X}$-invariant.  Note furthermore that in case~$\lambda_1(B) = \lambda_{n-k}(B)$ the test statistic~$T_B$ is constant everywhere on~$\R^n$. Therefore, such a choice of~$B$ is uninteresting for practical purposes. Note also that assigning the value~$\lambda_1(B)$ (instead of any other value) to the test statistic on~$\limfunc{span}(X)$ has no effect on rejection probabilities, because~$P_{\beta, \sigma, \rho}$ is absolutely continuous w.r.t.~$\mu_{\R^n}$ for every~$\beta \in \R^k$,~$\sigma \in (0, \infty)$ and~$\rho \in [0, a)$, and~$\limfunc{span}(X)$ being of dimension~$k < n$ implies~$\mu_{\R^n}(\limfunc{span}(X)) = 0$.

The following remark discusses two particularly important choices of~$B$:

\begin{remark}\label{rem:opttest}
Under regularity conditions and excluding degenerate cases, point-optimal invariant (w.r.t.~$G_X$) tests and locally best invariant (w.r.t.~$G_X$) tests for the testing problem~\eqref{testproblem} reject for large values of a test statistic~$T_B$ as in Equation~\eqref{T_quadratic}:
\begin{enumerate}[(a)]
\item Point-optimal invariant tests against the alternative~$\bar{\rho} \in (0, a)$ are obtained for~$B=-\left( C_{X}\Sigma (\bar{\rho})C_{X}^{\prime }\right)^{-1}$.
\item Locally best invariant tests are obtained for~$B=C_{X}%
\dot{\Sigma}(0)C_{X}^{\prime }$, for~$\dot{\Sigma}(0)$
the derivative of~$\Sigma$ at~$\rho =0$, ensured to exist under the aforementioned regularity conditions, see, e.g., \cite{KingHillier1985}. 
\end{enumerate}
Note that a test statistic~$T_B$ based on any of the two matrices~$B$ in the preceding enumeration does \emph{not} depend on the specific choice of~$C_X$, as any two choices of~$C_X$ differ only by premultiplication of an orthogonal matrix. However, for matrices~$B$ of a different form than~(a) or~(b) the test statistic~$T_{B}$ may also depend on the choice of~${C_{X}}$, a dependence which is typically suppressed in our notation.
\end{remark}

The main focus of the present article concerns power properties of tests based on a test statistic as in~\eqref{T_quadratic} for the testing problem~\eqref{testproblem}. Before investigating power properties of a test, one needs to ensure that its size does not exceed a given value of significance~$\alpha$. While this can be a nontrivial problem in general, achieving size control through the choice of a proper critical value turns out to be an easy task here. More specifically, the following lemma shows that exact size control for tests based on a test statistic~$T_B$ introduced in Equation~\eqref{T_quadratic} is possible at all levels of significance in the leading case~$\lambda_1(B) < \lambda_{n-k}(B)$. The subsequent remark discusses numerical aspects.

\begin{lemma}\label{lem:sizecontrol}
Let~$B \in \R^{(n-k)\times(n-k)}$ be symmetric and such that~$\lambda_1(B) < \lambda_{n-k}(B)$. Then, there exists a (unique) function~$\kappa: [0, 1] \to [\lambda_1(B), \lambda_{n-k}(B)]$ such that for every~$\alpha \in [0, 1]$
\begin{equation}
P_{\beta, \sigma, 0}\left(\Phi_{B, \kappa(\alpha)}\right) = \alpha \quad \text{for every } \beta \in \R^k \text{ and every } \sigma \in (0, \infty).
\end{equation}
Furthermore,~$\kappa$ is a strictly decreasing and continuous bijection.
\end{lemma}

\begin{remark}\label{rem:numericcv}
The rejection probabilities of a~$G_X$-invariant test for \eqref{testproblem} do not depend on the parameters~$\beta$ and~$\sigma$ (cf.~Remark~2.3 in \cite{PP17}). As a consequence, the exact critical value~$\kappa(\alpha)$ from Lemma~\ref{lem:sizecontrol} can easily be obtained numerically: To this end one can exploit the well-known fact that for every~$c \in \R$ the rejection probability~$P_{\beta, \sigma, 0}(\Phi_{B, c}) = P_{0, 1, 0}(\Phi_{B, c})$ can be rewritten as the probability that the quadratic form
\begin{equation}
\mathbf{G}' \left[ B - c I_{n-k}   \right] \mathbf{G} > 0,
\end{equation}
where~$\mathbf{G}$ is an~$(n-k)$-variate Gaussian random vector with mean zero and covariance matrix~$I_{n-k}$. This probability can be determined efficiently through an application of standard algorithms, e.g., the algorithm by \cite{davies1980algorithm}. The critical value~$\kappa(\alpha)$ can then be obtained numerically by simply using a root-finding algorithm to determine the unique root~$\kappa(\alpha)$ of~$c \mapsto P_{0, 1, 0}(\Phi_{B, c}) - \alpha$ on~$[\lambda_1(B), \lambda_{n-k}(B)]$.
\end{remark}

\section{The zero-power trap in testing for correlation} \label{sec:ZPT}

\subsection{Definition and sufficient conditions}

In the sequel, a test~$\varphi:\R^n \to [0, 1]$ (measurable) for testing problem~\eqref{testproblem} is said to be subject to (or suffer from) the zero-power trap, if there exist~$\beta \in \R^k$ and~$\sigma \in (0, \infty)$ such that
\begin{equation}\label{ZPT}
\liminf_{\rho \to a} E_{\beta, \sigma, \rho}(\varphi) = 0;
\end{equation}
that is, if the power function of~$\varphi$ can get arbitrarily close to~$0$ as the strength of the correlation in the data, measured in terms of~$\rho$, increases. Recall from Remark~\ref{rem:numericcv} that if~$\varphi$ is~$G_X$-invariant, which is the case for most tests considered in this article, then~$E_{\beta, \sigma, \rho}(\varphi)$ does not depend on~$\beta$ and~$\sigma$. In this case, if Equation~\eqref{ZPT} holds for some~$\beta \in \R^k$ and some~$\sigma \in (0, \infty)$, it holds for every~$\beta \in \R^k$ and every~$0 < \sigma < \infty$. 

\smallskip

A set of sufficient conditions that allows one to conclude whether a test~$\varphi$ is subject to the zero-power trap was developed in \cite{mart10} and \cite{PP17}. The underlying effect leading to \eqref{ZPT} described in the latter article is a concentration effect in the (rescaled) distribution~$P_{\beta, \sigma, \rho}$ when~$\rho$ is close to~$a$. \cite{PP17} obtained their sufficient conditions under the following property of the function~$\Sigma$ (cf.~also Assumption~1 in \cite{PP17} and the discussion there showing that this condition is weaker than the one previously used by \cite{mart10}):
\begin{assumption}\label{ASC}
$\lambda_n^{-1}(\Sigma(\rho)) \Sigma(\rho) \to ee'$ as~$\rho \to a$ for some~$e \in \R^n$.
\end{assumption}
%
For the convenience of the reader and for later use, we shall now formally state two immediate consequences of results in \cite{PP17}. They provide sufficient conditions for the zero-power trap under Assumption~\ref{ASC}. Specializing Theorem~2.7 and Remark~2.8 in \cite{PP17} one obtains the following ``high-level''-result.

\begin{theorem}\label{thm:ZPThl}
Suppose Assumption~\ref{ASC} holds. Let~$\varphi$ be a~$G_X$-invariant test that is continuous at~$e$ and satisfies~$\varphi(e) = 0$, where~$e$ is the vector figuring in Assumption~\ref{ASC}. Then 
\begin{equation}\label{eqn:ZPTPP}
\lim_{\rho \to a} E_{\beta, \sigma, \rho}(\varphi) = 0 \quad \text{ for every } \beta \in \R^k \text{ and every } \sigma \in (0, \infty).
\end{equation}
In particular, if~$\varphi = \mathbf{1}_W$ holds for some~$G_X$-invariant Borel set~$W \subseteq \R^n$, then~\eqref{eqn:ZPTPP} holds if~$e$ is not in the closure of~$W$.
\end{theorem}

For the test with rejection region~$\Phi_{B, \kappa(\alpha)}$ as discussed in Section~\ref{sec:tests} and where~$\kappa(\alpha)$ is defined through Lemma~\ref{lem:sizecontrol} one obtains the following result from Corollary~2.21 
of \cite{PP17}.

\begin{theorem}\label{lem:zptphiPP17}
Suppose Assumption~\ref{ASC} holds and~$e \notin \limfunc{span}(X)$, where~$e$ is the vector figuring in Assumption~\ref{ASC}. Let~$B \in \R^{(n-k)\times(n-k)}$ be symmetric and such that~$\lambda_1(B) < \lambda_{n-k}(B)$. Then,
%
for every~$\alpha \in (0, 1)$ such that~$T_{B}(e) < \kappa(\alpha)$ we have
\begin{equation}\label{eqn:limp}
\lim_{\rho \to a} P_{\beta, \sigma, \rho}\left(\Phi_{B, \kappa(\alpha)}\right) = 0 \quad \text{ for every } \beta \in \R^k \text{ and every } \sigma \in (0, \infty). 
\end{equation}
%
\end{theorem}

Note that the sufficient conditions for the zero-power trap phenomenon pointed out in Theorems~\ref{thm:ZPThl} and~
\ref{lem:zptphiPP17} depend on observable quantities only, and that they are thus \emph{checkable} by the user. Therefore, a researcher interested in testing problem~\eqref{testproblem} can use these conditions to check whether or not the given test suffers from the zero-power trap before actually using a test. In particular, one can decide \emph{not} to use a test that suffers from the zero-power trap. Before addressing the question \emph{how to avoid} the zero-power trap, which was raised already in the Introduction, we briefly pay some attention to the following question: ``how often'' does the zero-power trap actually arise? More specifically, in the important class of tests~$\Phi_{B, c}$ introduced in Section~\ref{sec:tests}, and most notably the tests discussed in Remark~\ref{rem:opttest}, the following question arises: For ``how many'' design matrices~$X$ does the zero-power trap arise? Answering this question is the content of the next section. 

\subsection{For ``how many'' design matrices does the zero-power trap arise?}\label{sec:sev}

We shall focus on the class of tests with rejection regions~$\Phi_{B(X), c}$ introduced in Section~\ref{sec:tests}. Since the question in the section title depends on the design matrix~$X$, which is otherwise held fixed in this article, we shall make the dependence of~$B$ on~$X$ explicit by writing~$B(X)$. Furthermore, we shall also write~$P_{\beta, \sigma, \rho}^X$ to emphasize its dependence on the design matrix~$X$. In our first attempt to answer the question under consideration, we shall use the following simple consequence of Lemma~\ref{lem:sizecontrol} and Theorem~\ref{lem:zptphiPP17}, which provides conditions on~$X$ under which Equation~\eqref{eqn:limp} holds for all ``small'' levels~$\alpha$.
\begin{lemma}\label{thm:conlcsuff}
Suppose Assumption~\ref{ASC} holds and let~$e$ denote the vector figuring in that assumption. Let~$B$ be a function from the set of full column rank~$n \times k$ matrices to the set of symmetric~$(n-k) \times (n-k)$-dimensional matrices. If an~$n \times k$ matrix~$X$ satisfies
\begin{equation}\label{eqn:concl}
\limfunc{rank}(X) = k \text{ and }C_Xe \notin \mathrm{Eig}\left(B(X), \lambda_{n-k}(B(X))\right),
\end{equation}
then~$\lambda_1(B(X)) < \lambda_{n-k}(B(X))$,~$P^{X}_{0, 1, 0}(\Phi_{B(X), T_{B(X)}(e)}) > 0$ and Equation~\eqref{eqn:limp} holds for every~$\alpha \in (0, P^{X}_{0, 1, 0}(\Phi_{B(X), T_{B(X)}(e)}))$.
\end{lemma}
For a class of functions~$X \mapsto B(X)$ that includes the ones discussed in Remark~\ref{rem:opttest} we shall now show that condition~\eqref{eqn:concl} is generically satisfied, unless the matrix~$B(X)$ has a very exceptional form. The result is established under a restriction concerning the eigenspace corresponding to the largest eigenvalue of~$B(X)$. 

\begin{proposition}\label{prop:smallalpha}
Suppose that~$k < n-1$ and that Assumption~\ref{ASC} holds. Let~$B$ be a function from the set of full column rank~$n \times k$ matrices to the set of symmetric~$(n-k) \times (n-k)$-dimensional matrices. Let~$M \in \R^{n \times n}$ be a symmetric matrix that can \emph{not} be written as~$c_1 I_n + c_2 ee'$ for real numbers~$c_1, c_2$ with~$c_2 \geq 0$, where~$e$ is the vector figuring in Assumption~\ref{ASC}. Suppose further that for every~$X \in \R^{n \times k}$ of full column rank a~$C_X \in \R^{(n-k) \times n}$ satisfying~$C_XC_X' = I_{n-k}$ and~$C_X'C_X = \Pi_{\limfunc{span}(X)^{\bot}}$ can be chosen such that  
\begin{equation}\label{eqn:asEig}
\limfunc{Eig}\left(B(X), ~\lambda_{n-k}(B(X))\right) = \limfunc{Eig}\left(C_{X}MC_{X}^{\prime }, ~\lambda_{n-k}(C_{X}MC_{X}^{\prime })\right).
\end{equation}
Then, up to a~$\mu_{\R^{n \times k}}$-null set of exceptional matrices, every~$X \in \R^{n \times k}$ satisfies~\eqref{eqn:concl}.
%
An immediate consequence is as follows: Given~$\alpha \in (0, 1)$ denote by~$\mathscr{X}(\alpha; B) \subseteq \R^{n \times k}$ the set of all~$X \in \R^{n \times k}$ of rank~$k$ such that~$\lambda_1(B(X)) < \lambda_{n-k}(B(X))$ and such that  
\begin{equation}
\lim_{\rho \to a} P^X_{\beta, \sigma, \rho}(\Phi_{B(X), C_X, \kappa(\alpha)}) = 0 \quad \text{ for every } \beta \in \R^k \text{ and every } \sigma \in (0, \infty).
\end{equation}
Then,~$\mathscr{X}(\alpha_2; B) \subseteq \mathscr{X}(\alpha_1; B)$ holds for~$0 < \alpha_1 \leq \alpha_2 < 1$, and for any sequence~$\alpha_m$ in~$(0, 1)$ converging to~$0$ the complement of~$\bigcup_{m \in \N} \mathscr{X}(\alpha_m; B)$ is contained in a~$\mu_{\R^{n \times k}}$-null set.
\end{proposition}

\begin{remark}\label{rem:eigmon}
Note that for~$B(X)=C_{X}%
\dot{\Sigma}(0)C_{X}^{\prime }$ Condition~\eqref{eqn:asEig} in Proposition~\ref{prop:smallalpha} is trivially satisfied with~$M = \dot{\Sigma}(0)$. For~$B(X) =-\left( C_{X}\Sigma (\bar{\rho})C_{X}^{\prime }\right) ^{-1}$ and~$\bar{\rho} \in (0, a)$ it is easy to see that Condition~\eqref{eqn:asEig} is satisfied with~$M = \Sigma(\bar{\rho})$. Therefore, if for any of these two specific choices the additional condition holds that the respective~$M$ can not be written as~$c_1 I_n + c_2 ee'$ for real numbers~$c_1, c_2$ where~$c_2 \geq 0$ is satisfied, then Proposition~\ref{prop:smallalpha} applies. 
\end{remark}

Proposition~\ref{prop:smallalpha} shows that tests based on~$T_{B(X)}$ suffer from the zero-power trap for ``most'' design matrices~$X$, at least for small choices of~$\alpha$. The discussion section of \cite{kramer85} contains a corresponding statement (without proof) in a special case. 

Choosing~$\alpha$ small is not completely uncommon in practice: Due to the fact that testing for correlation is often just one part of the econometric analysis, the actual level~$\alpha$ employed in this test can be quite small. One example is specification testing. Another example is the situation where tests for correlation are ``inverted'' to build a confidence interval for~$\rho$, which is then used for a Bonferroni-type construction of a data-dependent critical value of another test (cf.~\cite{Leeb2017} for further information concerning such critical values).

Nevertheless, the question remains as to how ``large'' the set~$\mathscr{X}(\alpha; B)$ actually is for a \emph{fixed}~$\alpha$, such as the conventional~$\alpha = .05$ or~$\alpha = .01$. For example, Proposition~\ref{prop:smallalpha} does not tell us whether or not the set of design matrices~$\mathscr{X}(.01; B)$ is empty. Similarly, one can ask if~$\mathscr{X}(.01; B)$ contains an open set, or if it has positive~$\mu_{\R^{n \times k}}$ measure? The latter questions have already been considered in detail in the main results of \cite{martellosio2012testing} for point-optimal invariant and locally best invariant tests in the important context of spatial autoregressive regression models. Adopting his proof strategy, we establish the following proposition. The argument requires a different assumption on~$B$ than the one used in Proposition~\ref{prop:smallalpha}. First, the condition used now concerns the eigenspace of~$B(X)$ corresponding to its \emph{smallest} eigenvalue (as opposed to the condition on the largest eigenvalue used in Proposition~\ref{prop:smallalpha}). Second, continuity conditions are imposed, which are required for limiting arguments in the proof. As discussed in Remark~\ref{rem:summary} below, the assumptions are again satisfied in the leading choices for~$B$ discussed in Remark~\ref{rem:opttest}.

\begin{proposition}\label{prop:mart}
Suppose that~$k < n-1$ and that Assumption~\ref{ASC} holds. Let~$B$ be a function from the set of full column rank~$n \times k$ matrices to the set of symmetric~$(n-k) \times (n-k)$-dimensional matrices. Suppose there exists a function~$F$ from the set of~$(n-k) \times (n-k)$ matrices to itself, such that for every~$X \in \R^{n \times k}$ of full column rank~$B(X) = F(C_X M C_X')$ holds for a suitable choice of~$C_X \in \R^{(n-k) \times n}$ satisfying~$C_XC_X' = I_{n-k}$ and~$C_X'C_X = \Pi_{\limfunc{span}(X)^{\bot}}$, and for~$M \in \R^{n \times n}$ a symmetric matrix that can \emph{not} be written as~$c_1 I_n + c_2 ee'$ for real numbers~$c_1, c_2$ where~$c_2 \geq 0$. Here~$e$ is the vector figuring in Assumption~\ref{ASC}. Suppose further that~$F$ is continuous at every element~$A$, say, of the closure of~$\{C_X MC_X': X \in \R^{n \times k},~\limfunc{rank}(X) = k\} \subseteq \R^{(n-k)\times (n-k)}$, and that for every such~$A$ we have
\begin{equation}\label{eqn:eigprops}
	\limfunc{Eig}\left(F(A), ~\lambda_{1}(F(A))\right) = \limfunc{Eig}\left(A, ~\lambda_{1}(A)\right).
\end{equation}
%
Define~$\mathscr{X}(\alpha; B) \subseteq \R^{n \times k}$ as in Proposition~\ref{prop:smallalpha}. Then, the following holds:
\begin{enumerate}
\item~$\mathscr{X}(\alpha; B) \neq \emptyset$ holds for every~$\alpha \in (0, 1)$;
\item suppose that for every~$z \in \R^n$ the function~$X \mapsto T_{B(X), C_X}(z)$ is continuous at every~$X \in \R^{n \times k}$ of full column rank such that~$z \notin \limfunc{span}(X)$. Then, for every~$\alpha \in (0, 1)$ the interior of $\mathscr{X}(\alpha; B)$ is nonempty (and thus has positive~$\mu_{\R^{n \times k}}$ measure).
\end{enumerate}
\end{proposition}

\begin{remark}\label{rem:summary}
Similar to Remark~\ref{rem:eigmon} we note that Proposition~\ref{prop:mart} can be applied to~$B(X) = C_X \dot{\Sigma}(0)C_X'$ (with~$M = \dot{\Sigma}(0)$ and~$F$ the identity function), or to~$B(X) = -(C_X \Sigma(\overline{\rho})C_X')^{-1}$, where~$\overline{\rho} \in (0, a)$, (with~$M = \Sigma(\overline{\rho})$ and~$F$ the function~$A \mapsto -A^{-1}$, noting that this function satisfies the continuity requirement as~$\Sigma(\overline{\rho})$ is positive definite) provided that the corresponding~$M$ matrix is not of the exceptional form~$c_1 I_n + c_2 ee'$ for~$c_2 \geq 0$. It is not difficult to show that the continuity requirement in Part~2 of the proposition is satisfied for these two choices of~$B$. For~$B(X) = C_X \dot{\Sigma}(0) C_X'$ this is trivial. For~$B(X) = -(C_X \Sigma(\overline{\rho})C_X')^{-1}$, where~$\overline{\rho} \in (0, a)$, an argument is given in Appendix~\ref{ZPTAPP}. We can hence conclude that unless~$\dot{\Sigma}(0)$ or~$\Sigma(\overline{\rho})$, respectively, is of the form~$c_1 I_n + c_2 ee'$ for some nonnegative~$c_2$, the test~$\Phi_{B(X), \kappa(\alpha)}$ suffers from the zero-power trap for every~$\alpha \in (0, 1)$ for every~$X$ in a non-empty open set of design matrices.
\end{remark}

\begin{remark}\label{rem:exc}
We emphasize that Propositions~\ref{prop:smallalpha} and~\ref{prop:mart} do \emph{not} apply in case~$M = c_1 I_n + c_2 ee'$ holds for real numbers~$c_1, c_2$ where~$c_2 \geq 0$. On the one hand, it is clear that in case~$c_2 = 0$ a test as in these two propositions with~$M = c_1 I_n$ trivially breaks down, as the corresponding test statistics are then constant. But on the other hand, as already observed (for the special case~$c_1 = 0$ and~$c_2 = 1$) in \cite{PP17} in the discussion preceding their Remark~2.27, using tests based on~$M = c_1 I_n + c_2 ee'$ for a~$c_2 > 0$ indeed presents an opportunity to avoid the zero power trap. This will be discussed more formally in Section~\ref{sec:ee}.
\end{remark}

From the results in the present section we learn that for tests that satisfy certain structural properties, the zero power trap arises for generic design matrices for~$\alpha$ small enough. Furthermore, for every~$\alpha$ there exists (under suitable assumptions) a nonempty open set of design matrices every element of which suffers from the zero-power trap. 
%
We would like to emphasize, however, that these results \emph{do not rule out} the possibility that for a given~$X$ the actual level~$\alpha$ needed such that the zero-power trap arises can be low (far outside the commonly used range of levels), or that given~$\alpha$ the open set of design matrices for which the zero-power trap occurs is ``small''. Numerical results that illustrate the ``practical severity'' of the zero-power trap in spatial regression models are provided in Section~3 of \cite{kramer2005finite}, in particular his Table~1 is very interesting in this context, and further discussion and examples can be found in \cite{mart10} and \cite{martellosio2012testing}. These results seem to suggest that the zero-power trap occurs frequently for commonly used levels of significance in case~$n-k$ is ``small'', i.e., in ``high-dimensional'' scenarios, whereas if~$n-k$ is large the zero-power trap does not appear that frequently. However, this also depends on the dependence structure.

\section{Avoiding the zero-power trap}\label{sec:avoid}

Having provided some context and motivation, we now discuss three ways to avoid the zero-power trap: In Section~\ref{sec:ee} we expand on the observation just made in Remark~\ref{rem:exc}. The strategy discussed in Section~\ref{sec:artreg} is based on an idea involving artificial regressors. The method we recommend, however, builds on Section~\ref{sec:ee} and is introduced in Section~\ref{subs:approx}. Our suggestion tries to overcome sub-optimality properties of the other methods. As discussed in the Introduction, the idea underlying our approach can be interpreted as a finite sample variant of the power enhancement principle of \cite{fan2015power}. 

\subsection{Tests based on~$T_B$ with~$B = C_X ee' C_X'$}\label{sec:ee}


As discussed in Remark~\ref{rem:exc}, tests based on the test statistic~$T_B$ with~$B = C_X ee' C_X'$ do not satisfy the assumptions underlying Propositions~\ref{prop:smallalpha} and~\ref{prop:mart}. Hence, these two propositions do not let us conclude anything concerning the question ``how often'' the zero-power trap occurs for such tests. It turns out that these tests do not suffer from the zero-power trap for any~$\alpha \in (0, 1)$ in case the additional condition~$e \notin \limfunc{span}(X)$ holds (note that if~$e \in \limfunc{span}(X)$ holds, the test statistic~$T_B$ with~$B = C_X ee' C_X'$ is useless as it equals~$0$ for every~$y \in \R^n$). As pointed out in Remark~\ref{rem:exc}, this was already noted in~\cite{PP17}. For later use in Section~\ref{subs:approx} we state a corresponding result (which is an immediate consequence of Part~1 of Proposition~2.26 in \cite{PP17} together with~$G_X$-invariance of~$T_B$ and our Lemma~\ref{lem:sizecontrol}):

\begin{theorem}\label{thm:ee}
Suppose that~$k < n-1$, that Assumption~\ref{ASC} holds and that~$e \notin \limfunc{span}(X)$, where~$e$ is the vector figuring in Assumption~\ref{ASC}. Then, for every~$\alpha \in (0, 1)$, every~$\beta \in \R^k$ and every~$\sigma \in (0, \infty)$
\begin{equation}
\lim_{\rho \to a} P_{\beta, \sigma, \rho}(\Phi_{C_X ee' C_X', \kappa(\alpha)}) = 1.
\end{equation}
\end{theorem}

From this result we conclude that in case~$e \notin \limfunc{span}(X)$ and whenever a test~$\varphi$ with size~$\alpha$ is subject to the zero-power trap, one can alternatively use the test with rejection region~$\Phi_{C_Xee'C_X', \kappa(\alpha)}$ instead, which does not suffer from the zero-power trap. Moreover, the power of the test~$\Phi_{C_Xee'C_X', \kappa(\alpha)}$ even increases to~$1$ as~$\rho \to a$. This is a desirable property as it matches the intuition that strong correlations should be easily detectable from the data.

\smallskip

While avoiding the zero-power trap problem, the test~$\Phi_{C_Xee'C_X', \kappa(\alpha)}$ suffers from one major disadvantage: the power function of~$\Phi_{C_Xee'C_X', \kappa(\alpha)}$ can be, and often will be, quite low for values~$\rho \in (0, a)$ distant from~$a$. If the initial test~$\varphi$, which was dismissed because it is subject to the zero-power trap, was chosen because of its good power properties in this region of the alternative, the test~$\Phi_{C_X ee'C_X', \kappa(\alpha)}$ will then not constitute a convincing alternative. This is illustrated in the example discussed in Section~\ref{sec:num}. A method that tries to take optimality properties of the initial test into account, at least for the classes of tests discussed in Remark~\ref{rem:opttest}, is discussed next.

\subsection{Tests based on artificial regressors}\label{sec:artreg}

The sufficient condition for the zero-power trap in Theorem~\ref{lem:zptphiPP17} requires that the vector~$e$ from Assumption~\ref{ASC} is \emph{not} an element of~$\limfunc{span}(X)$. While this of course does not \emph{prove} that the zero-power trap does not arise if~$e \in \limfunc{span}(X)$, this indeed turns out to be the case under an additional assumption (cf.~Corollary~2.22 in \cite{PP17}). In this section we shall exploit this fact. The method of avoiding the zero-power trap we discuss in this section ``enforces'' the condition~$e \in \limfunc{span}(X)$. More specifically, it is based on adding the vector~$e$ from Assumption~\ref{ASC} as an ``artificial'' regressor to the design matrix (if it is not already an element of~$\limfunc{span}(X)$), and from then constructing tests as if this artificially expanded design matrix was the true one. As discussed in the Introduction, the idea underlying the construction in the present section can be traced back to \cite{kramer85}.


\smallskip

To formally describe the artificial regressor based method in our general setting, consider a situation where a  researcher initially wants to use the test~$\Phi_{B, \kappa(\alpha)}$ as in Section~\ref{sec:tests} with~$\lambda_1(B) < \lambda_{n-k}(B)$, but discovers (e.g., by checking the sufficient conditions in Theorem~\ref{lem:zptphiPP17}) that~$\Phi_{B, \kappa(\alpha)}$ suffers from the zero-power trap. Suppose further that the initial test~$\Phi_{B, \kappa(\alpha)}$ has certain optimality properties (cf.~Remark~\ref{rem:opttest}). The researcher does not want to completely sacrifice the optimality properties of the initial test, which prevents him from using the test just discussed in Section~\ref{sec:ee}. Assume further that~$e \notin \limfunc{span}(X)$.

The trick now is to work with the design matrix~$\bar{X} = (X, e)$ in the construction of a test statistic, assuming that~$k+1 < n$. More precisely, let~$\bar{B}$ be a symmetric~$(n - k - 1) \times (n - k - 1)$ matrix (cf.~Remark~\ref{rem:intbarB} below), and define the adjusted test statistic
\begin{equation}\label{eqn:adjtest}
\bar{T}_{\bar{B}}\left( y\right) =\bar{T}_{\bar{B},C_{\bar{X}}}\left( y\right) =
\begin{cases}
y^{\prime }C_{\bar{X}}^{\prime }\bar{B}C_{\bar{X}}y/\Vert {C_{\bar{X}}y}\Vert ^{2} & \text{if }%
y\notin \limfunc{span}(\bar{X}) \\ 
\lambda _{1}(\bar{B}) & \text{if }y\in \limfunc{span}(\bar{X}).%
\end{cases}
\end{equation}
Under the additional assumption that~$\lambda_1(\bar{B}) < \lambda_{n-k-1}(\bar{B})$, one obtains\footnote{To obtain this statement one needs to apply Lemma~\ref{lem:sizecontrol} to model \eqref{linmod} but with design matrix~$\bar{X}$ instead of~$X$. Note that this leads to an ``enlarged'' model that encompasses the true model as a submodel; and that the distributions satisfying the null hypothesis in the true model also satisfy the null hypothesis in the enlarged model.} from Lemma~\ref{lem:sizecontrol} for every~$\alpha \in (0, 1)$ the existence and uniqueness of a critical value~$\bar{\kappa}(\alpha) \in (\lambda_1(\bar{B}), \lambda_{n-k-1}(\bar{B}))$, say, such that for every~$\beta \in \R^k$ and every~$\sigma \in (0, \infty)$ it holds that
\begin{equation}\label{eqn:corsize}
P_{\beta, \sigma, 0}(\{y \in \R^n: \bar{T}_{\bar{B}}(y) > \bar{\kappa}(\alpha) \}) = \alpha.
\end{equation}
Finally, define the rejection region
\begin{equation}\label{eqn:rejadj}
\bar{\Phi}_{\bar{B}, \bar{\kappa}(\alpha)} := \{y \in \R^n: \bar{T}_{\bar{B}}(y) > \bar{\kappa}(\alpha) \}.
\end{equation}

\begin{remark}\label{rem:intbarB}
We think about~$\bar{B}$ as an ``updated version'' of~$B$, i.e., as the matrix one would use if~$\bar{X}$ was the underlying design matrix. For example, if the initial matrix~$B$ equals~$C_X\dot{\Sigma}(0) C_X'$ one could use~$\bar{B} = C_{\bar{X}}\dot{\Sigma}(0) C_{\bar{X}}'$, or if the initial matrix~$B = -(C_X\Sigma(\bar{\rho}) C_X')^{-1}$ one could use~$\bar{B} = -(C_{\bar{X}}\Sigma(\bar{\rho}) C_{\bar{X}}')^{-1}$. Recall that the rejection region \eqref{eqn:rejadj} based on these two versions of~$\bar{B}$ corresponds to locally best invariant tests and point-optimal invariant tests, respectively, in the model where the true design matrix is~$\bar{X}$ (cf.~Remark~\ref{rem:opttest}). 
\end{remark}
We shall now prove that the test with rejection region \eqref{eqn:rejadj} does not suffer from the zero-power trap. The following result requires an additional assumption on~$\Sigma(.)$. This is Assumption~4 in \cite{PP17} to which we refer the reader for equivalent formulations, examples and further discussion.

\begin{assumption}\label{as:add}
There exists a function~$c : [0,a)\to(0,\infty)$, a normalized vector~$e \in \R^n$, and a square root~$L_*(.)$ of~$\Sigma(.)$ such that
\begin{equation}
\Lambda := \lim_{\rho \to a} c(\rho) \Pi_{\limfunc{span}(e)^{\bot}} L_*(\rho)
\end{equation}
exists in~$\R^{n\times n}$ and such that the linear map~$\Lambda$ is injective when restricted to~$\limfunc{span}(e)^{\bot}$.
\end{assumption}

The main result concerning artificial regressor based tests is as follows:

\begin{theorem}\label{thm:artreg}
Suppose Assumptions~\ref{ASC} and~\ref{as:add} are satisfied with the same vector~$e$, that~$e \notin \limfunc{span}(X)$, and that~$k<n-1$. Suppose further that~$\bar{B}$ is a symmetric~$(n-k-1) \times (n-k-1)$ matrix such that~$\lambda_1(\bar{B}) < \lambda_{n-k-1}(\bar{B})$. Then, for every~$\alpha \in (0, 1)$, every~$\beta \in \R^k$ and every~$\sigma \in (0, \infty)$ it holds that
\begin{equation}
0 < \lim_{\rho \to a} P_{\beta, \sigma, \rho}(\bar{\Phi}_{\bar{B}, \bar{\kappa}(\alpha)}) = \mathrm{Pr}\left(\bar{T}_{\bar{B}}(\Lambda \mathbf{G}) > \bar{\kappa}(\alpha)\right) < 1,
\end{equation}
where~$\mathbf{G}$ denotes a Gaussian random vector with mean~$0$ and covariance matrix~$I_n$.
\end{theorem}
Theorem~\ref{thm:artreg} shows that~$\bar{\Phi}_{\bar{B}, \bar{\kappa}(\alpha)}$ is not subject to the zero-power trap. However, its ``limiting power''~$\lim_{\rho \to a} P_{\beta, \sigma, \rho}(\bar{\Phi}_{\bar{B}, \bar{\kappa}(\alpha)}) = \mathrm{Pr}(\bar{T}_{\bar{B}}(\Lambda \mathbf{G}) > \bar{\kappa}(\alpha))$ can in principle be low. In particular, it is always smaller than one. This is different to the behavior of the test discussed in Section~\ref{sec:ee}, which has limiting power equal to one. Another limitation of Theorem~\ref{thm:artreg} is its reliance on the additional Assumption~\ref{as:add}.

Following up on the examples discussed in Remark~\ref{rem:intbarB}, an advantage of passing from~$\Phi_{B, \kappa(\alpha)}$ to~$\Phi_{\bar{B}, \bar{\kappa}(\alpha)}$, instead of passing from~$\Phi_{B, \kappa(\alpha)}$ to the test discussed in Section~\ref{sec:ee}, is that~$\Phi_{\bar{B}, \bar{\kappa}(\alpha)}$ ``preserves'' in some sense the optimality properties of~$\Phi_{B, \kappa(\alpha)}$, but with respect to the larger group~$G_{\bar{X}}$. Note, however, that this does \emph{not} imply that the power functions of~$\Phi_{B, \kappa(\alpha)}$ and~$\bar{\Phi}_{\bar{B}, \bar{\kappa}(\alpha)}$ are ``close''. 


\subsection{Optimality-preserving tests that avoid the zero-power trap}\label{subs:approx}

The starting point in this section is an (initial) family of tests~$\varphi_{\alpha}: \R^n \to [0, 1]$ for the testing problem~\eqref{testproblem} indexed by~$\alpha \in (0, 1)$. Given~$\alpha \in (0, 1)$ we interpret~$\varphi_{\alpha}$ as the (initial) test one would like to use because of some optimality property. That is, the power function of~$\varphi_{\alpha}$~$$(\beta, \sigma, \rho) \mapsto E_{\beta, \sigma, \rho}(\varphi_{\alpha})$$ is ``large'' for certain parameter values~$(\beta, \sigma, \rho)$ in a given subset pertaining to the alternative hypothesis~$\{0\} \times (0, \infty) \times (0, a)$. 

We shall suppose that the initial test~$\varphi_{\alpha}$ suffers from the zero-power trap, which one would like to avoid. Ideally a test should have limiting power equal to~$1$, a property of the test in Section~\ref{sec:ee}, but not of the test in Section~\ref{sec:artreg}. Furthermore, we would like to keep, at least approximately, the optimal power properties of~$\varphi_{\alpha}$, which was the reason why~$\varphi_{\alpha}$ was considered for use initially. This is a property of the test in Section~\ref{sec:artreg} (at least to some extent), but not of the test in Section~\ref{sec:ee}. We shall now present an approach that achieves these two goals.

\smallskip

In what follows, we assume that the family of tests~$\{\varphi_\alpha\}$ under consideration satisfies Property~A, i.e., satisfies the following:

\medskip

\begin{addmargin}[1em]{2em}%
\begin{enumerate}[{A}.1:]
\item For every~$\alpha \in (0, 1)$ the test~$\varphi_{\alpha}$ is~$G_X$-invariant.
\item For every~$\alpha \in (0, 1)$ the test~$\varphi_{\alpha}$ has size~$\alpha$, i.e.,~$$\sup_{\beta \in \R^k} \sup_{\sigma \in (0, \infty)} E_{\beta, \sigma, 0}(\varphi_{\alpha}) = \alpha.$$ 
\item For every~$\alpha \in (0, 1)$ and every sequence~$\alpha_m \in [0, \alpha]$ converging to~$\alpha$ we have that~$\varphi_{\alpha_m}(y) \to \varphi_{\alpha}(y)$ holds for~$\mu_{\R^n}$-almost every~$y\in \R^n$.
\end{enumerate}
\end{addmargin}

\bigskip

To illustrate the assumption, consider the following important example: 
\begin{example}\label{rem:exenh}
Let~$T_{B}$ be as in~\eqref{T_quadratic} with~$B$ an~$(n-k)\times (n-k)$ symmetric matrix such that~$\lambda_{1}(B) < \lambda_{n-k}(B)$. For every~$\alpha \in (0, 1)$ let~$\kappa(\alpha)$ be the critical value from Lemma~\ref{lem:sizecontrol}. Set~$\varphi_{\alpha}$ equal to the non-randomized test with rejection region~$\Phi_{B, \kappa(\alpha)}$, i.e.,~$\varphi_{\alpha} := \mathbf{1}_{ \Phi_{B, \kappa(\alpha)} }$. We already know that~$T_B$ is~$G_X$-invariant, and thus~$\varphi_{\alpha}$ is~$G_X$-invariant for every~$\alpha$. Hence~A.1 is satisfied. Furthermore, from Lemma~\ref{lem:sizecontrol} we see that~$\varphi_{\alpha}$ satisfies~A.2. That~A.3 is satisfied is an immediate consequence of continuity of~$\kappa(.)$, which was established in Lemma~\ref{lem:sizecontrol}, together with the fact that for every~$\alpha \in (0, 1)$ the set
\begin{equation}
\{ y \in \R^n: T_B(y) = \kappa(\alpha)\}
\end{equation}
is a~$\mu_{\R^n}$-null set; the latter is a consequence of Lemma~B.4 in \cite{PP17}, which shows that the cdf.~$F$, say, corresponding to~$P_{0, 1, 0} \circ T_{B}$ is continuous.
\end{example}

\begin{remark}\label{rem:ascending}
While not required in Property~A, typical families~$\{\varphi_{\alpha}\}$ will also satisfy the condition that for any real numbers~$\alpha_1 \leq \alpha_2$ in~$(0, 1)$ it holds for~$\mu_{\R^n}$-almost every~$y\in \R^n$ that~$\varphi_{\alpha_1}(y) \leq \varphi_{\alpha_2}(y)$. For instance, this is the case for the families of tests discussed in Example~\ref{rem:exenh} (this follows from the monotonicity property of~$\kappa(.)$ established in Lemma~\ref{lem:sizecontrol}).  One obvious consequence of this condition is that if~$\varphi_{\alpha_2}$ suffers from the zero-power trap, then~$\varphi_{\alpha_1}$ suffers from the zero-power trap as well. Therefore, for such families, if~$\varphi_{\alpha}$ suffers from the zero-power trap, there is no hope that one can easily avoid the zero-power trap by using~$\varphi_{\alpha - \varepsilon}$ for some~$\varepsilon > 0$ (which would at least be a test whose size does not exceed~$\alpha$).
\end{remark}

Suppose in the following discussion that~$k<n-1$, that Assumption~\ref{ASC} holds and that~$e \notin \limfunc{span}(X)$. Recall from Theorem~\ref{thm:ee} that under these conditions the~$G_X$-invariant test $\Phi_{C_Xee'C_X', \kappa(\alpha)}$ does not suffer from the zero-power trap, in fact has limiting power one, at all levels~$\alpha \in (0, 1)$. Using this property, we shall now define a~$G_X$-invariant test that has approximately the same power properties of~$\varphi_{\alpha}$ with the advantage that it has limiting power~$1$ just as the test~$\Phi_{C_Xee'C_X', \kappa(\alpha)}$. 

The basic idea is as follows (precise statements are provided further below): From Property~A.3 one obtains that for~$\varepsilon \in (0, \alpha)$ small, the power functions of~$\varphi_{\alpha}$ and~$\varphi_{\alpha - \varepsilon}$ are similar. Theorem~\ref{thm:ee} tells us that the test with rejection region~$\Phi_{C_Xee'C_X', \kappa(\varepsilon)}$ has limiting power (as~$\rho \to a$) equal to~$1$, and Lemma~\ref{lem:sizecontrol} shows that this test has size equal to~$\varepsilon$. Hence, we \emph{could} use the~$G_X$-invariant test~
\begin{equation}\label{eqn:couldusethis}
\min(\varphi_{\alpha - \varepsilon} + \mathbf{1}_{\Phi_{C_Xee'C_X', \kappa(\varepsilon)}}, 1),
\end{equation} whose power function is similar to~$\varphi_{\alpha}$ (at least for~$\varepsilon$ small), but which has limiting power equal to one (for every~$0 < \varepsilon < \alpha$). Trivially, this test has size not greater than~$\alpha$, but potentially its size is smaller than~$\alpha$, implying some unnecessary loss in power, which one can try to avoid by decreasing~$\kappa(\varepsilon)$. 

More specifically, define the~$G_X$-invariant test
\begin{equation}\label{eqn:enhtest}
\varphi^*_{\alpha, \varepsilon} := \min \left(\varphi_{\alpha - \varepsilon} + \mathbf{1}_{\Phi_{C_Xee'C_X', c(\alpha, \varepsilon)}},~ 1\right) = \varphi_{\alpha - \varepsilon} + (1-\varphi_{\alpha - \varepsilon})\mathbf{1}_{\Phi_{C_Xee'C_X', c(\alpha, \varepsilon)}},
\end{equation}
where~$0 < c(\alpha, \varepsilon) \leq \kappa(\varepsilon)$ is chosen to be the \emph{smallest} number such that~$\varphi^*_{\alpha, \varepsilon}$ has size \emph{equal} to~$\alpha$. That such a choice of~$c(\alpha, \varepsilon)$ is indeed possible is the content of the next proposition. Note that~$\varphi^*_{\alpha, \varepsilon}$ is non-randomized if the test~$\varphi_{\alpha - \varepsilon}$ is non-randomized. 
\begin{proposition}\label{prop:excrit}
Suppose that~$k < n-1$, that~$e \in \R^n$ satisfies~$e \notin \limfunc{span}(X)$, and that the family~$\{\varphi_{\alpha}\}$ satisfies Properties~A.1 and~A.2. Then, for every~$\alpha \in (0, 1)$ and every~$\varepsilon \in (0, \alpha)$ there exists a~$c(\alpha, \varepsilon) \in (0, \kappa(\varepsilon)]$ such that 
\begin{equation}\label{eqn:defc}
\sup_{\beta \in \R^k} \sup_{\sigma \in (0, \infty)} E_{\beta, \sigma, 0}\left[ \min \left(\varphi_{\alpha - \varepsilon} + \mathbf{1}_{\Phi_{C_Xee'C_X', c(\alpha, \varepsilon)}},~ 1\right) \right] =  \alpha,
\end{equation}
and such that for every~$c' \in (0, c(\alpha, \varepsilon))$ it holds that the supremum in the previous display is greater than~$\alpha$; here~$\kappa(\varepsilon) \in (0, \|C_Xe\|^2)$ denotes the unique real number such that~$\Phi_{C_Xee'C_X', \kappa(\varepsilon)}$ has size equal to~$\varepsilon$ (cf.~Lemma~\ref{lem:sizecontrol}).
\end{proposition}

Note that the critical value~$c(\alpha, \varepsilon)$ can be easily determined numerically by a simple line search algorithm,~cf.~also Remark~\ref{rem:numericcv}.

Having established that the test in Equation~\eqref{eqn:enhtest} is actually well-defined, we now prove that it does not suffer from the zero-power trap but has limiting power~$1$ for any choice of~$\varepsilon$. Furthermore, we show that the power function of~$\varphi^*_{\alpha, \varepsilon}$ approximates (even uniformly over suitable subsets of the parameter space) the power function of~$\varphi_{\alpha}$ as~$\varepsilon$ converges to~$0$. In this sense, choosing~$\varepsilon > 0$ small, the test~$\varphi^*_{\alpha, \varepsilon}$ preserves ``optimal'' power properties (such as point-optimal invariance, or locally best invariance, cf.~Example~\ref{rem:exenh} above) from~$\varphi_{\alpha}$ at least approximately. Furthermore, the degree of approximation can be tuned by the user via~$\varepsilon$.

\begin{theorem}\label{thm:enhthm}
Suppose that~$k < n-1$, that Assumption~\ref{ASC} holds and that~$e \notin \limfunc{span}(X)$, where~$e$ is the vector figuring in Assumption~\ref{ASC}. Assume that the family~$\{\varphi_{\alpha}\}$ satisfies Properties~A.1 and~A.2. Let~$\alpha \in (0, 1)$. Then, the following holds:
\begin{enumerate}
\item For every~$\varepsilon \in (0, \alpha)$, every~$\beta \in \R^k$ and every~$\sigma \in (0, \infty)$ we have~$$\lim_{\rho \to a} E_{\beta, \sigma, \rho}(\varphi^*_{\alpha, \varepsilon}) = 1;$$ in particular~$\varphi^*_{\alpha, \varepsilon}$ does \emph{not} suffer from the zero-power trap.
\item Suppose that the family~$\{\varphi_{\alpha}\}$ also satisfies Property~A.3. Let~$A \subseteq [0, a)$ be such that the closure of the set
\begin{equation}
\{\Sigma(\rho)/\|\Sigma(\rho)\|: \rho \in A\}
\end{equation}
is contained in the set of positive definite symmetric matrices. Then
\begin{equation}
\lim_{\varepsilon \to 0^+} \sup_{\beta \in \R^k} \sup_{\sigma \in (0, \infty)} \sup_{\rho \in A} |E_{\beta, \sigma, \rho}(\varphi^*_{\alpha, \varepsilon}) - E_{\beta, \sigma, \rho}(\varphi_{\alpha})| = 0.
\end{equation}
\end{enumerate}
\end{theorem}

\begin{remark}
In the leading case~$\Sigma(.)$ is a continuous function. In this case one can choose the set~$A$ in the second part of Theorem \ref{thm:enhthm} equal to~$[0, c]$ for any~$0 < c < a$ [recall that~$\Sigma(\rho)$ is positive definite for every~$\rho \in [0, a)$ by assumption]. Note further that since we are primarily interested in situations where the initial test~$\varphi_{\alpha}$ suffers from the zero-power trap, while the adjusted tests~$\varphi_{\alpha, \varepsilon}^*$ have limiting power~$1$, it is not restrictive to confine ourselves to intervals~$[0, c]$ as above, as we do \emph{not} want the power of the adjusted test to be close to the power of the initial test in a neighborhood of~$a$. Furthermore, the optimality properties of point-optimal invariant tests (against an alternative~$\bar{\rho} \in (0, a)$) or of locally best invariant tests (which are characterized by favorable power properties in the neighborhood of~$0$) concern only the power function over~$[0, c]$ for a suitably chosen~$c < a$.
\end{remark}

\begin{remark}
The tuning parameter~$\varepsilon$ needs to be chosen by the user in each particular application. In principle, the user can plot the power functions for various values of~$\varepsilon$, and can then decide upon inspection, which value of~$\varepsilon$ provides the best solution. For a specific example we refer to Section~\ref{sec:num} below.
\end{remark}

\begin{remark}
Finally, we point out that the construction of~$\varphi^*_{\alpha, \varepsilon}$ in Equation~\eqref{eqn:enhtest} and the conditions in Proposition~\ref{prop:excrit} and  Theorem~\ref{thm:enhthm} do not require the initial test~$\varphi_{\alpha}$ to suffer from the zero-power trap. While this is clearly our main focus, this observation shows that our method can also be applied in case the limiting-power of~$\varphi_{\alpha}$ is greater than~$0$ but smaller than one. In such a situation, using~$\varphi_{\alpha, \varepsilon}^*$ instead of~$\varphi_{\alpha}$ can be advantageous as well.
\end{remark}

\section{Numerical results}\label{sec:num}

In order to illustrate and compare the power properties of the tests introduced in Section~\ref{sec:avoid}, we now consider a simple example from spatial econometrics in which the zero-power trap occurs for a popular test. We focus on a situation where the correlation between the observations is a consequence of their proximity, which might be spatial, but could also be, e.g., social, and which is encoded in the adjacency (``weights'') matrix of a graph. 

One important model in this case is the spatial (autoregressive) error model, which leads to $$\Sigma(\rho) = [(I-\rho W')(I-\rho W)]^{-1},$$ for~$W$ a fixed weights matrix which is assumed to be (elementwise) nonnegative and irreducible with zero elements on the main diagonal. By the Perron-Frobenius theorem (e.g., \cite{HJ1985}, Theorem~8.4.4), the matrix~$W$ then has a positive (real) eigenvalue~$\lambda_{\max}(W)$, say, with algebraic multiplicity (and thus also geometric multiplicity) equal to~1, such that any other real or complex zero of the characteristic polynomial of~$W$ is in absolute value not larger than~$\lambda_{\max}(W)$. We assume that the parameter~$\rho \in [0, \lambda_{\max}(W)^{-1})$. For~$f_{\max}$ a normalized eigenvector of~$W$ w.r.t.~$\lambda_{\max}(W)$ it is not too difficult to see that Assumption~\ref{ASC} is satisfied (with~$e = f_{\max}$), and that Assumption~\ref{as:add} is satisfied. For details we refer to Section 4.1 in \cite{PP17}. 

The model depends, besides the design matrix~$X$, on the specific form of the weights matrix~$W$, which encodes the dependence relation of the observations. Subsequently we reconsider a simple example considered in Section 3 of \cite{kramer2005finite}, who has observed (cf.~his Figure~1) that for a weights matrix derived by the Queen criterion from a~$4 \times 4$ regular lattice, and for~$X = (1, \hdots, 1)' \in \R^{16}$ the Cliff-Ord test suffers from the zero-power trap for~$\alpha = 5\%$. We recall that the Cliff-Ord test is based on a test statistic as in Equation~\eqref{T_quadratic} and with~$B = C_X(W + W')C_X'$. 

The power function of the Cliff-Ord test and the power functions of the tests described in Section~\ref{sec:avoid} were obtained numerically (cf.~also Remark~\ref{rem:numericcv}), and are shown in Figure~\ref{fig:powqueen}. The figure also shows the power \emph{envelope} in the class of~$G_X$-invariant tests. That is, for each alternative~$\bar{\rho} \in (0, \lambda_{\max}(W)^{-1})$ Figure~\ref{fig:powqueen} shows the power of the point-optimal~$G_X$-invariant level~$\alpha = 5\%$ test against the alternative~$\bar{\rho}$. Recall from Remark~\ref{rem:opttest} that the point-optimal invariant test against alternative~$\bar{\rho}$ is based on a test statistic as in~\eqref{T_quadratic} and with~$B = -[C_X\Sigma(\bar{\rho})C_X']^{-1}$. In this example the power envelope is not attained by any~$G_X$-invariant test, but it serves the purpose of providing an upper bound for comparison. 

While Figure~\ref{fig:powqueen} illustrates that the approaches discussed in Sections~\ref{sec:ee} and~\ref{sec:artreg} avoid the zero-power trap, it reveals at the same time that the power functions of these tests are not completely satisfying. On the one hand, even though the test introduced in Section~\ref{sec:ee} does not suffer from the zero-power trap, it has low power in a large region of the alternative. On the other hand, the test from Section~\ref{sec:artreg} based on the Cliff-Ord test (i.e., as in Equation \eqref{eqn:rejadj} with~$\bar{B} = C_{(X, e)} (W + W') C_{(X, e)}'$) with artificial regressor~$e = f_{\max}$ avoids the zero-power trap as well and has a power function that practically coincides with the power envelope for small values of~$\rho$. But its limiting power is smaller than one (in fact is only~$0.619$).

\begin{figure}[!ht]
\centering
\includegraphics[width=0.8\linewidth]{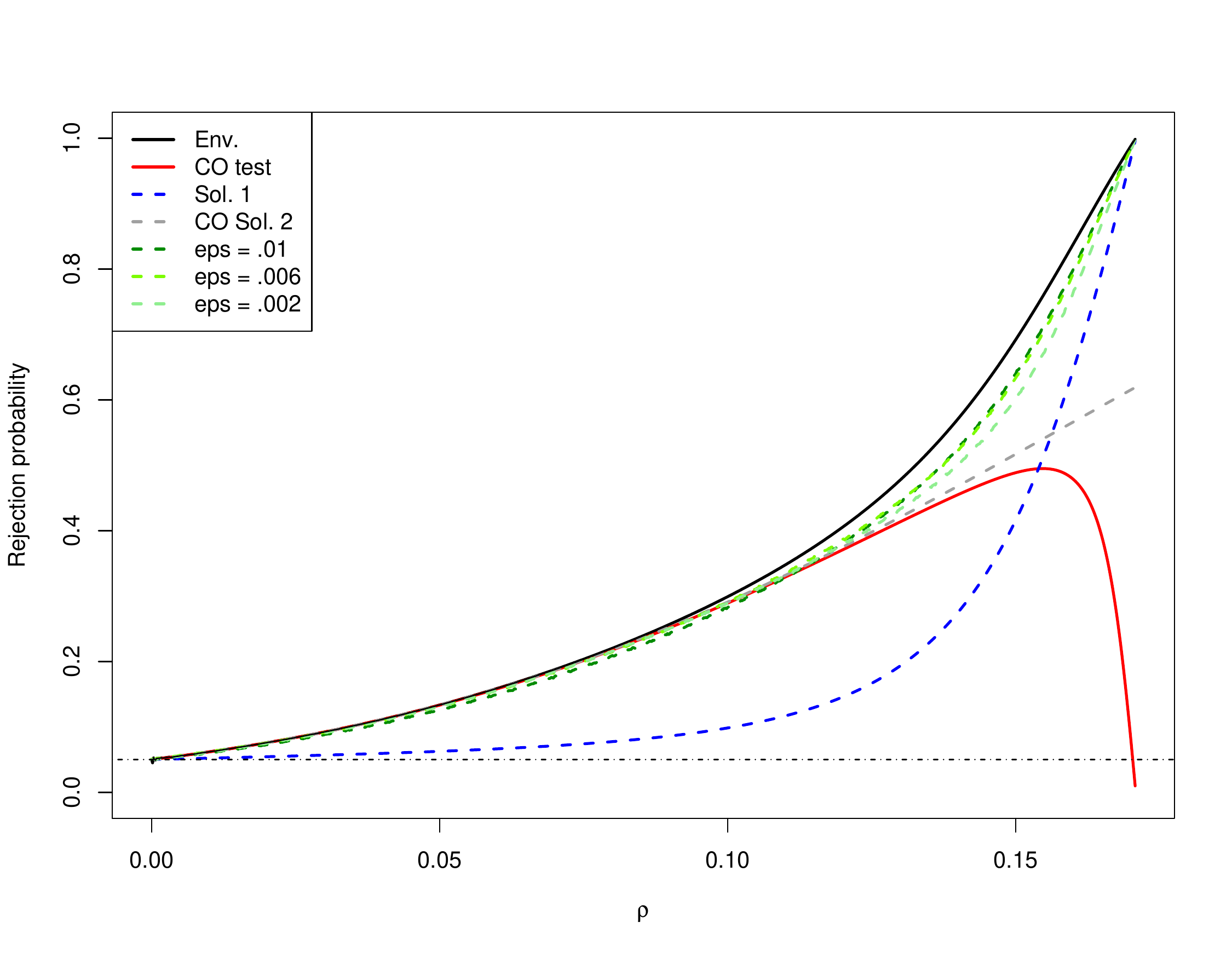}
\caption{Comparison of power functions. The horizontal line corresponds to~$\alpha = .05$. Env.~corresponds to the power envelope; CO test to the power function of the Cliff-Ord test; Sol.~1 to the power function of the test from Section~\ref{sec:ee}; CO Sol.~2 to the power function of the artificial regressor based Cliff-Ord test as discussed in Section~\ref{sec:artreg};~$\text{eps} = .01$ corresponds to the test in Section~\ref{subs:approx} with~$\varphi_{\alpha}$ the size~$\alpha$ Cliff-Ord test and~$\varepsilon = .01$; the remaining~$\text{eps} = .006$ and~$\text{eps} = .002$ correspond likewise to the tests in Section~\ref{subs:approx},  but for the corresponding values of~$\varepsilon$.}
\label{fig:powqueen}
\end{figure}

Figure~\ref{fig:powqueen} also contains the power function of some tests corresponding to the procedure outlined in Section~\ref{subs:approx} applied to the family~$\varphi_{\alpha}$ of level-$\alpha$ Cliff-Ord tests (cf.~Example~\ref{rem:exenh}). It shows the power functions corresponding to~$\varepsilon \in \{.002, .006, .01\}$. These tests have very good power properties. The power functions are practically identical to the one of the Cliff-Ord test (and hence to the power envelope) for small values of~$\rho$. But for larger values of~$\rho$ their power function is much closer to the power envelope than the power of the Cliff-Ord test. In particular, by construction, their power converges to~$1$ as~$\rho$ gets close to~$a$. One can also observe that smaller values of~$\varepsilon$ lead to power functions that are closer to the power function of the Cliff-Ord test for~$\rho$ close to~$0$, whereas larger values of~$\varepsilon$ lead to power functions that are closer to the power envelope for~$\rho$ close to~$a$.

\section{Conclusion}\label{sec:concl}

In the present article we have re-considered the zero-power trap phenomenon in testing for correlation in a general framework. Most importantly, we have suggested a way to construct ``approximately optimal tests'' that avoid the trap. For practical purposes, if an initial test, such as the Cliff-Ord test in the example discussed in Section~\ref{sec:num}, turns out to suffer from the zero-power trap, we suggest to use the method introduced in Section~\ref{subs:approx} to obtain a modified test with the following properties: (i) it has a similar power function as the initial test, (ii) it does not suffer from the zero-power trap, and (iii) its limiting power equals one. The tuning parameter~$\varepsilon$ involved in the construction of the modified test can be chosen by graphically comparing the power functions of modified tests corresponding to different values of the tuning parameter with the power envelope and the power function of the initial test. The heuristic underlying our construction can be interpreted as a finite sample variant of the power enhancement principle of \cite{fan2015power}. The approach, which is not restricted to the testing problem under consideration, might be of some interest in its own right.

\bibliographystyle{chicago}	
\bibliography{refs}		

\begin{center}
\begin{LARGE}
\textbf{Appendices}
\end{LARGE}
\end{center}

\begin{appendix}
	
\section{Proofs for results in Section~\ref{sec:intro}} \label{app:intro}

\begin{proof}[Proof of Lemma~\ref{lem:sizecontrol}:]
Lemma~B.4 in \cite{PP17} shows that the cdf.~$F$, say, corresponding to~$P_{0, 1, 0} \circ T_{B}$ is continuous, that~$F(\lambda_1(B)) = 0$,~$F(\lambda_{n-k}(B)) = 1$, and that~$F$ is strictly increasing on~$[\lambda_1(B), \lambda_{n-k}(B)]$. Hence, the function~$f: [\lambda_1(B), \lambda_{n-k}(B)] \to [0, 1]$ defined via
\begin{equation}
c \mapsto P_{0, 1, 0}(\Phi_{B, c}) = 1-F(c)
\end{equation}
is continuous, strictly decreasing, and satisfies~$f(\lambda_1(B)) = 1$ and~$f(\lambda_{n-k}(B)) = 0$. Set~$\kappa = f^{-1}$, i.e., the inverse of~$f$, which is continuous, strictly decreasing, and obviously satisfies~$\kappa(0) = \lambda_{n-k}(B)$ and~$\kappa(1) = \lambda_{1}(B)$. Then,~$P_{0, 1, 0}\left(\Phi_{B, \kappa(\alpha)}\right) = \alpha$ for every~$\alpha \in [0, 1]$. Finally, recall that~$T_{B}$ is~$G_X$-invariant, from which it follows (cf.~Remark~2.3 in \cite{PP17}) that for every~$c\in \R$ every~$\beta \in \R^k$ and every~$\sigma \in (0, \infty)$ we have~$P_{\beta, \sigma, 0}(\Phi_{B, c}) = P_{0, 1, 0}(\Phi_{B,  c}).$ Hence,~$P_{\beta, \sigma, 0}\left(\Phi_{B, \kappa(\alpha)}\right) = \alpha$ holds for every~$\beta \in \R^k$, every~$\sigma \in (0, \infty)$, and every~$\alpha \in [0, 1]$. The uniqueness part is obvious.
\end{proof}
	
\section{Proofs for results in Section~\ref{sec:ZPT}}\label{ZPTAPP}

\begin{proof}[Proof of Theorem~\ref{thm:ZPThl}:]
We apply Theorem~2.7 in \cite{PP17}. Their Assumption 1 coincides with ours and is thus satisfied. Furthermore, by our Gaussianity assumption, their Assumption~3 is satisfied in our framework (with~$\mathbf{z}$ a normally distributed random vector with mean~$0$ and covariance matrix~$I_n$), and we can use Part~1 of their Proposition~2.6 to conclude that their Assumption~2 is satisfied. The statement now follows from Theorem~2.7 in \cite{PP17} for the special case~$\varphi(e) = 0$. The last statement follows from Remark~2.8(i) in the same reference.
\end{proof}

\begin{proof}[Proof of Theorem~\ref{lem:zptphiPP17}:]
We use Corollary~2.21 in \cite{PP17}. That their Assumptions~1 and~2 are satisfied follows as in the proof of Theorem~\ref{thm:ZPThl} above. Recall from Lemma~\ref{lem:sizecontrol} that~$\kappa$ is a strictly decreasing and continuous bijection from~$[0, 1]$ to~$[\lambda_1(B), \lambda_{n-k}(B)]$, implying that for~$\alpha \in (0, 1)$ we have~$\kappa(\alpha) \in (\lambda_1(B), \lambda_{n-k}(B))$. We can hence apply Corollary~2.21 in \cite{PP17} to conclude that (under our assumptions) for~$\alpha \in (0, 1)$ such that~$T_{B}(e) < \kappa(\alpha)$ we have~\eqref{eqn:limp}.
\end{proof}

\begin{proof}[Proof of Lemma~\ref{thm:conlcsuff}:]
Noting that both~$e \notin \limfunc{span}(X)$ and~$\lambda_{1}(B(X)) < \lambda_{n-k}(B(X))$ follow from~$C_Xe \notin \limfunc{Eig}(B(X), \lambda_{n-k}(B(X)))$, Condition~\eqref{eqn:concl} together with the definition of~$T_{B(X)}$ in Equation~\eqref{T_quadratic} can be used to verify~$\lambda_1(B(X)) \leq T_{B(X)}(e) < \lambda_{n-k}(B(X))$.  Thus, Lemma~\ref{lem:sizecontrol} gives~$P_{0, 1, 0}^X(\Phi_{B(X), T_{B(X)}(e)}) \in (0, 1]$ and~$T_{B(X)}(e) < \kappa(\alpha)$ for every~$\alpha \in (0, P_{0, 1, 0}^X(\Phi_{B(X), T_{B(X)}(e)}))$. We can now apply Theorem~\ref{lem:zptphiPP17} to conclude.
\end{proof}

\begin{lemma}\label{lem:algebraic}
Let~$M \in \R^{n \times n}$ be symmetric, let~$v \in \R^n$ be such that~$\|v\| = 1$, and suppose that~$1 \leq d < n-1$. Then, 
\begin{equation*}
\mathscr{D}(n,d) := \{L \in \R^{n \times d}: \limfunc{rank}(L) = d, ~ \Pi_{\limfunc{span}(L)^{\bot}}v \text{ is an eigenvector of } \Pi_{\limfunc{span}(L)^{\bot}}M\Pi_{\limfunc{span}(L)^{\bot}}\}
\end{equation*}
can be written as
\begin{equation}
\{L \in \R^{n \times d}: \det(L'L) \neq 0,~\|\Pi_{\limfunc{span}(L)^{\bot}}v\| \neq 0 \} \cap \{L \in \R^{n \times d}: p_M(L) = 0 \}.
\end{equation}
for~$p_M: \R^{n \times d} \to \R$ a multivariate polynomial, which is given in the proof. Furthermore,~$p_M \equiv 0$ if and only if~$M = c_1 I_n + c_2 vv'$ holds for real numbers~$c_1$ and~$c_2$.
\end{lemma}

\begin{proof}
Let~$L \in \R^{n \times d}$ satisfy~$\limfunc{rank}(L) = d$, or equivalently~$\limfunc{det}(L'L) \neq 0$. If~$\Pi_{\limfunc{span}(L)^{\bot}}v = 0$, the vector~$\Pi_{\limfunc{span}(L)^{\bot}}v$ can not be an eigenvector of~$\Pi_{\limfunc{span}(L)^{\bot}} M \Pi_{\limfunc{span}(L)^{\bot}}$. If~$\Pi_{\limfunc{span}(L)^{\bot}}v \neq 0$,~$\Pi_{\limfunc{span}(L)^{\bot}}v$ is an eigenvector of the symmetric matrix~$\Pi_{\limfunc{span}(L)^{\bot}} M \Pi_{\limfunc{span}(L)^{\bot}}$ if and only if~
\begin{equation}\label{eqn:rankcondnew}
\limfunc{rank}\left((\Pi_{\limfunc{span}(L)^{\bot}} v, \Pi_{\limfunc{span}(L)^{\bot}} M \Pi_{\limfunc{span}(L)^{\bot}}v) \right) < 2.
\end{equation}
We can write this rank condition equivalently as
\begin{equation}\label{eqn:proofcond}
0 = \limfunc{det}\left[
(\Pi_{\limfunc{span}(L)^{\bot}} v, \Pi_{\limfunc{span}(L)^{\bot}} M \Pi_{\limfunc{span}(L)^{\bot}}v )'
(\Pi_{\limfunc{span}(L)^{\bot}} v, \Pi_{\limfunc{span}(L)^{\bot}} M \Pi_{\limfunc{span}(L)^{\bot}}v )   \right].
\end{equation}
Writing~$\Pi_{\limfunc{span}(L)^{\bot}} = I_n - \det(L'L)^{-1} L\limfunc{adj}(L'L)L'$ (throughout we use the convention that the adjoint of a~$1\times 1$ matrix equals~$1$), and premultiplying~\eqref{eqn:proofcond} by~$\det(L'L)^{16} \neq 0$, one sees that~\eqref{eqn:proofcond} is equivalent to
$$0 = \limfunc{det}\left[
(\det(L'L)Q(L) v, Q(L) M Q(L)v )'
(\det(L'L)Q(L) v, Q(L) M Q(L)v )   \right]=:p_M(L),$$
where~$Q(L) := \det(L'L) I_n - L\limfunc{adj}(L'L)L'$. Note that~$L \mapsto p(L)$ defines a multivariate polynomial on~$\R^{n \times d}$. It follows that~$\mathscr{D}(n,d)$ has the claimed form. 

To prove the second statement, note that if~$M$ is of the specific form~$c_1 I_n + c_2 vv'$ for real numbers~$c_1$ and~$c_2$, one has for every~$L \in \R^{n \times d}$ that
\begin{equation}
\Pi_{\limfunc{span}(L)^{\bot}} M \Pi_{\limfunc{span}(L)^{\bot}} v = (c_1 + c_2 v'\Pi_{\limfunc{span}(L)^{\bot}}v) \Pi_{\limfunc{span}(L)^{\bot}} v.
\end{equation}
For~$L$ such that~$\det(L'L) \neq 0$ the statement~$p_M(L) = 0$ is equivalent to~\eqref{eqn:rankcondnew}. But~\eqref{eqn:rankcondnew} holds because of the previous display. If~$L$ satisfies~$\det(L'L) = 0$ we obviously have~$p_M(L) = 0$. Thus,~$p_M \equiv 0$ for all~$M$ of this specific form.

Now assume that~$M$ can not be written as~$c_1 I_n + c_2 vv'$ for real numbers~$c_1$ and~$c_2$. It suffices to construct a single~$L$ such that~$p_M(L) \neq 0$ holds. We consider two cases:

\smallskip

(a) We first show that one can find an~$L$ as required in the special case where~$v$ is not an eigenvector of~$M$. Let~$u_1, \hdots, u_n$ be an orthonormal basis of eigenvectors of~$M$ with corresponding eigenvalues~$\lambda_{1}(M), \hdots, \lambda_n(M)$. Note that there then exist two indices~$j \neq l$, say, such that~$\lambda_{j}(M) \neq \lambda_{l}(M)$ and such that~$v'u_j \neq 0$ and~$v'u_l \neq 0$ (otherwise~$v$ would be an eigenvector of~$M$; recall that~$v \neq 0$). Now, define the matrix~$L_{\bot} = (u_j, u_l, z_1, \hdots, z_{n-d-2})$ for~$z_1, \hdots, z_{n-d-2}$ linearly independent elements of~$\limfunc{span}(u_j, u_l, v)^{\bot}$ (with the convention that~$L_{\bot} = (u_j, u_l)$ if~$n-d = 2$; note that~$n-d \geq 2$ holds by assumption). Such a choice of~$z_1, \hdots, z_{n-d-2}$ is possible as~$d \geq 1$ by assumption. Note that~$\limfunc{rank}(L_{\bot}) = n-d$. Next, let~$L$ be an~$n \times d$ matrix with~$\limfunc{span}(L)=\limfunc{span}(L_{\bot})^{\bot}$. Then,~$L$ is of full column rank, and~$\Pi_{\limfunc{span}(L)^{\bot}}v \neq 0$. From the discussion preceding the definition of~$p_M$ we see that it thus remains to verify that~$\Pi_{\limfunc{span}(L)^{\bot}} v$ is not an eigenvector of~$\Pi_{\limfunc{span}(L)^{\bot}} M \Pi_{\limfunc{span}(L)^{\bot}}$. But~$\Pi_{\limfunc{span}(L)^{\bot}} v = \Pi_{\limfunc{span}((u_j, u_l))} v = u_j'v u_j + u_l'v u_l$, implying~$\Pi_{\limfunc{span}(L)^{\bot}} M \Pi_{\limfunc{span}(L)^{\bot}}v = \lambda_j(M) u_j'v u_j + \lambda_l(M) u_l'v u_l$. Hence, if~$\Pi_{\limfunc{span}(L)^{\bot}} v$ was an eigenvector of~$\Pi_{\limfunc{span}(L)^{\bot}} M \Pi_{\limfunc{span}(L)^{\bot}}$, we would have
\begin{equation}
\lambda_j(M) u_j'v u_j + \lambda_l(M) u_l'v u_l = c(u_j'v u_j + u_l'v u_l)
\end{equation}
for some~$c \in \R$, which gives the contradiction~$\lambda_j(M) = \lambda_l(M) = c$. 

\smallskip

(b) Next we consider the case where~$v$ is an eigenvector of~$M$ to the eigenvalue~$\lambda_{i}(M)$, say. Let~$u_1, \hdots, u_n$ be an orthonormal basis of eigenvectors of~$M$ corresponding to its eigenvalues~$\lambda_1(M), \hdots, \lambda_n(M)$, and where~$u_i = v$ holds. By assumption,~$M$ is not of the form~$c_1 I_n + c_2 vv'$. Together with~$v$ being an eigenvector of~$M$ this implies (via a diagonalization argument) existence of two indices~$j$ and~$l$, say, such that~$i,j,l$ are pairwise distinct and such that~$\lambda_j(M) \neq \lambda_l(M)$. Now, define~$L_{\bot} = (x, y, z_1, \hdots, z_{n-d-2})$ where~$x = v + u_j$,~$y = v + u_l$ and where~$z_1, \hdots, z_{n-d-2}$ are linearly independent elements of~$\limfunc{span}(u_j, u_l, v)^{\bot}$ (with the convention that~$L_{\bot} = (x, y)$ if~$n-d = 2$; recall that~$n-d \geq 2$ holds by assumption). Such a construction is possible as~$d\geq 1$ by assumption. Note that~$\limfunc{rank}(L_{\bot}) = n-d$. Define~$L$ as an~$n \times d$ matrix with~$\limfunc{span}(L) = \limfunc{span}(L_{\bot})^{\bot}$. Then,~$L$ is of full column rank, and~$\Pi_{\limfunc{span}(L)^{\bot}} v \neq 0$. Arguing as in (a) it now remains to verify that~$\Pi_{\limfunc{span}(L)^{\bot}} v$ is not an eigenvector of~$\Pi_{\limfunc{span}(L)^{\bot}} M \Pi_{\limfunc{span}(L)^{\bot}}$: It is easy to see that~$$\Pi_{\limfunc{span}(L)^{\bot}}v = \Pi_{\limfunc{span}((x,y))}v = 3^{-1}(x+y),$$ and that, using the expression in the previous display and a simple computation,~$$\Pi_{\limfunc{span}(L)^{\bot}} M \Pi_{\limfunc{span}(L)^{\bot}}v = 9^{-1} \left[ (2\lambda_i(M) + 2\lambda_j(M) - \lambda_l(M))x + (2\lambda_i(M) -\lambda_j(M) + 2\lambda_l(M))y \right].$$ Hence, for this choice of~$L$ the vector~$\Pi_{\limfunc{span}(L)^{\bot}}v$ is an eigenvector of~$\Pi_{\limfunc{span}(L)^{\bot}} M \Pi_{\limfunc{span}(L)^{\bot}}$ if and only if
\begin{equation}\label{eqn:newlasteqconclu}
3^{-1}(x+y) = c~9^{-1} \left[ (2\lambda_i(M) + 2\lambda_j(M) - \lambda_l(M))x + (2\lambda_i(M) + 2\lambda_l(M) -\lambda_j(M))y \right]
\end{equation}
for some~$c \in \R$. The number~$c$ must then necessarily be nonzero. But this implies (premultiply both sides of~\eqref{eqn:newlasteqconclu} by~$u_j'$, then by~$u_l'$, and compare the two equations obtained) that~$\lambda_j(M) = \lambda_l(M)$, a contradiction.
\end{proof}

\begin{proof}[Proof of Proposition~\ref{prop:smallalpha}:]
We start with the claim that up to a~$\mu_{\R^{n \times k}}$-null set of exceptional matrices, every~$X \in \R^{n \times k}$ satisfies~\eqref{eqn:concl}. From~$k<n$ it follows that~$\mu_{\R^{n \times k}}(\{X \in \R^{n \times k}:\mathrm{rank}(X) < k\}) = 0$. Hence, it suffices to show that
\begin{equation}\label{eqn:setX}
\{X \in \R^{n \times k}: \limfunc{rank}(X) = k \text{ and } C_Xe \in \mathrm{Eig}(B(X), \lambda_{n-k}(B(X)))\}
\end{equation}
is a~$\mu_{\R^{n \times k}}$-null set. We consider two cases:

\smallskip

(a) Suppose first that~$M = c_1 I_n + c_2 ee'$ for real numbers~$c_1, c_2$ where~$c_2 < 0$. Then, the set in Equation~\eqref{eqn:setX} simplifies to 
\begin{equation}\label{eqn:firstattempt}
\{X \in \R^{n \times k}: \limfunc{rank}(X) = k \text{ and } e \in \limfunc{span}(X)\}.
\end{equation}
To see this note that in this case and for~$X \in \R^{n \times k}$ so that~$\limfunc{rank}(X) = k$ we have
\begin{equation}\label{eqn:botstruc}
\limfunc{Eig}\left(B(X), ~\lambda_{n-k}(B(X))\right) = \limfunc{Eig}\left(C_{X}MC_{X}^{\prime }, ~\lambda_{n-k}(C_{X}MC_{X}^{\prime })\right) = \limfunc{span}(C_x e)^{\bot},
\end{equation}
where we used the assumption in~\eqref{eqn:asEig} to obtain the first equality, and the specific structure of~$M$ and~$c_2 < 0$ to obtain the second equality. Thus,~$C_Xe \in \limfunc{Eig}\left(B(X), ~\lambda_{n-k}(B(X))\right)$ is possible only if~$C_X e = 0$, which is equivalent to~$e \in \limfunc{span}(X)$. Therefore,~\eqref{eqn:setX} simplifies to~\eqref{eqn:firstattempt}. But, by assumption~$1 \leq k < n$ holds, from which it is easy to see, noting that~$\|e\| = 1$, that~$\mu_{\R^{n \times k}}(\{X \in \R^{n \times k}:  e \in \limfunc{span}(X)\}) = 0$. Therefore, the set in~\eqref{eqn:firstattempt}, and equivalently the set in~\eqref{eqn:setX}, is a~$\mu_{\R^{n \times k}}$-null set in this case.

\smallskip

(b) Consider now the case where~$M$ is not a linear combination of~$I_n$ and~$ee'$. Using Equation~\eqref{eqn:asEig} we can write the set defined in~\eqref{eqn:setX} equivalently as
\begin{equation}\label{eqn:set2}
\{X \in \R^{n \times k}: \limfunc{rank}(X) = k  \text{ and } C_Xe \in \limfunc{Eig}\left(C_{X}MC_{X}^{\prime }, ~\lambda_{n-k}(C_{X}MC_{X}^{\prime })\right)\}.
\end{equation}
For~$X \in \R^{n \times k}$ of full column rank the property~$C_X'C_X = \Pi_{\limfunc{span}(X)^{\bot}} = \Pi_{\limfunc{span}(X)^{\bot}}^2$ can be used to verify that~$$C_Xe \in \mathrm{Eig}(C_XMC_X', \lambda_{n-k}(C_XMC_X'))$$ implies~$$\Pi_{\limfunc{span}(X)^{\bot}}e \in \mathrm{Eig}(\Pi_{\limfunc{span}(X)^{\bot}}M\Pi_{\limfunc{span}(X)^{\bot}}, \lambda_{n-k}(C_XMC_X')).$$ Thus, if~$e \notin \limfunc{span}(X)$ then ~$\Pi_{\limfunc{span}(X)^{\bot}}e\neq 0$, and~$C_Xe \in \mathrm{Eig}(C_XMC_X', \lambda_{n-k}(C_XMC_X'))$ implies that~$\Pi_{\limfunc{span}(X)^{\bot}}e$ is an eigenvector of~$\Pi_{\limfunc{span}(X)^{\bot}}M\Pi_{\limfunc{span}(X)^{\bot}}$. Thus, the set in Equation~\eqref{eqn:set2} is contained in the union of the~$\mu_{\R^{n \times k}}$-null set~$\{X \in \R^{n \times k}: e \in \limfunc{span}(X) \}$ and the set
\begin{equation}\label{eqn:2ndset}
\{X \in \R^{n \times k}: \limfunc{rank}(X) = k, ~ \Pi_{\limfunc{span}(X)^{\bot}}e \text{ is an eigenvector of } \Pi_{\limfunc{span}(X)^{\bot}}M\Pi_{\limfunc{span}(X)^{\bot}}\}.
\end{equation}
It thus remains to verify that the set in~\eqref{eqn:2ndset} is a~$\mu_{\R^{n \times k}}$-null set. Lemma~\ref{lem:algebraic} (applied with~$k = d$ and~$v = e$) shows that~\eqref{eqn:2ndset} is the subset of an algebraic set. Note that the assumptions in Lemma~\ref{lem:algebraic} are satisfied as~$1 \leq k < n-1$ is assumed. The lemma also provides the information that a multivariate polynomial defining this algebraic set does not vanish everywhere. Hence, it follows that the set in the previous display is contained in a~$\mu_{\R^{n \times k}}$-null set. Since the set is Borel measurable (cf., e.g., the representation obtained via Lemma~\ref{lem:algebraic}), it follows that it is itself a~$\mu_{\R^{n \times k}}$-null set.

We now prove the two remaining claims concerning~$\mathscr{X}(\alpha; B)$. For the monotonicity claim: If~$\mathscr{X}(\alpha_2; B)$ is empty, there is nothing to prove. Consider the case where~$\mathscr{X}(\alpha_2; B) \neq \emptyset$. Let~$X \in \mathscr{X}(\alpha_2; B)$. By definition of~$\mathscr{X}(\alpha_2; B)$ the matrix~$X$ has full column rank and~$\lambda_1(B(X)) < \lambda_{n-k}(B(X))$. From~$0 < \alpha_1 \leq \alpha_2 < 1$ it thus follows from Lemma~\ref{lem:sizecontrol} that~$\kappa(\alpha_2) \leq \kappa(\alpha_1)$. Hence,~$\Phi_{B(X), C_X, \kappa(\alpha_1)} \subseteq \Phi_{B(X),C_X, \kappa(\alpha_2)}$ and one obtains~$X \in \mathscr{X}(\alpha_1; B)$. Finally, note that Lemma~\ref{thm:conlcsuff} shows that~if~$X$ satisfies \eqref{eqn:concl}, then~$X \in \bigcup_{m \in \N} \mathscr{X}(\alpha_m; B)$. The first (already established) part of the current proposition hence proves the last claim.
\end{proof}

\begin{lemma}\label{lem:exseq}
Let~$M \in \R^{n \times n}$ be symmetric, let~$v \in \R^n$ such that~$\|v\|=1$, and suppose that~$M$ can \emph{not} be written as~$c_1I_n + c_2vv'$ for real numbers~$c_1, c_2$ where~$c_2 \geq 0$. Let~$d \in \N$ such that~$d < n-1$. Then:
\begin{enumerate}
\item There exists a sequence~$L_m \in \R^{n \times d}$ such that~$L_m'L_m = I_d$ and~$L_m \to L^*$ as~$m \to \infty$, a vector~$u \in \R^n$ with~$\|u\| = 1$ and a real number~$c > \lambda_{\min}(M)$, such that:~$\Pi_{\limfunc{span}(L_m)^{\bot}}v \neq 0$ and~$\Pi_{\limfunc{span}(L_m)^{\bot}}u \neq 0$ holds for every~$m \in \N$, such that
\begin{equation}
\quad \label{eqn:eigconv}
\lim_{m \to \infty} v'\Pi_{\limfunc{span}(L_m)^{\bot}}M\Pi_{\limfunc{span}(L_m)^{\bot}}v/v'\Pi_{\limfunc{span}(L_m)^{\bot}}v = \lambda_{\min}(M), 
\end{equation}
and such that for every~$m \in \N$ we have
\begin{equation}
\label{eqn:maxconv}
\quad  u'\Pi_{\limfunc{span}(L_m)^{\bot}}M\Pi_{\limfunc{span}(L_m)^{\bot}}u/u'\Pi_{\limfunc{span}(L_m)^{\bot}}u = c.
\end{equation}
\item Let~$B$ be a function from the set of full column rank~$n \times d$ matrices to the set of symmetric~$(n-d) \times (n-d)$-dimensional matrices. Suppose there exists a function~$F$ from the set of~$(n-d) \times (n-d)$ matrices to itself, such that for every~$L \in \R^{n \times d}$ of full column rank~$B(L) = F(C_LM C_L')$ holds for  a suitable choice of~$C_L \in \R^{(n-d) \times n}$ satisfying~$C_LC_L' = I_{n-d}$ and~$C_L'C_L = \Pi_{\limfunc{span}(L)^{\bot}}$. Suppose further that~$F$ is continuous at every element~$A$, say, of the closure of~$\{C_L MC_L': L \in \R^{n \times d},~\limfunc{rank}(L) = d\} \subseteq \R^{(n-d) \times (n-d)}$, and that for every such~$A$ we have
\begin{equation}\label{eqn:eigenspaceeq}
\limfunc{Eig}\left(F(A), ~\lambda_{1}(F(A))\right) = \limfunc{Eig}\left(A, ~\lambda_{1}(A)\right).
\end{equation}
Then, the sequence~$L_m$ obtained in Part 1 satisfies~$C_{L_m}v \neq 0$ for every~$m \in \N$, 
\begin{equation}
\label{eqn:eigconv2}
\lim_{m \to \infty} \left[v'C'_{L_m} B(L_m) C_{L_m}v/\|C_{L_m}v\|^2 - \lambda_{1}(B(L_m)) \right] = 0, 
\end{equation}
and 
\begin{equation}
\label{eqn:maxconv2}
\liminf_{m \to \infty}  \left[ \lambda_{n-k}(B(L_m)) - \lambda_{1}(B(L_m)) \right] = \delta
\end{equation}
for some positive real number~$\delta$.
\end{enumerate}
\end{lemma}

\begin{proof}
Before we prove Part 1, we note that it suffices to verify the existence claim without the requirement that~$L_m$ converges: Convergence of~$L_m$ can then be achieved by passing to a subsequence.

1.a) Consider first the case where~$v \in \limfunc{Eig}(M, \lambda_{\min}(M))$: Let~$u \in \limfunc{Eig}(M, \lambda_{\max}(M))$ such that~$\|u\| = 1$, and set~$L_{m,\bot} := (u,v,w_1, \hdots, w_{n-d-2})$ for~$w_1, \hdots, w_{n-d-2}$ linearly independent elements of~$\limfunc{span}((u,v))^{\bot}$ (with the implicit understanding that~$L_{m,\bot} = (u,v)$ in case~$d  = n-2$). By assumption~$M$ is not a multiple of~$I_n$, thus~$\lambda_{\min}(M) < \lambda_{\max}(M)$, from which it also follows that~$L_{m, \bot}$ has full column rank~$n-d\geq 2$ for every~$m \in \N$. For every~$m \in \N$ set~$L_m$ equal to an~$n \times d$ matrix such that~$L_m'L_m = I_d$ and~$\limfunc{span}(L_m)^{\bot} = \limfunc{span}(L_{m, \bot})$. Then Equations~\eqref{eqn:eigconv} and~\eqref{eqn:maxconv} (with~$c = \lambda_{\max}(M)$) follow immediately from~$\Pi_{\limfunc{span}(L_m)^{\bot}}v = v$ and~$\Pi_{\limfunc{span}(L_m)^{\bot}}u = u$.

1.b) Next, we consider the case where~$v \notin \limfunc{Eig}(M, \lambda_{\min}(M))$: We first claim that there must exist an~$x \in \limfunc{Eig}(M, \lambda_{\min}(M))$ such that~$\|x\| = 1$ and a vector~$u \in \limfunc{span}(v,x)^{\bot}$ such that~$\|u\| = 1$ and such that~$u'Mu > \lambda_{\min}(M)$. We argue by contradiction: First of all, if the claim was false, then~$\limfunc{dim}(\limfunc{Eig}(M, \lambda_{\min}(M))) = n-1$ would follow. 
We could then choose~$v_1, \hdots, v_{n-1}$ an orthonormal basis of~$\limfunc{Eig}(M, \lambda_{\min}(M))$.  Under the assumption that the above claim was wrong, it would further follow that~$\limfunc{span}(v, v_i)^{\bot} \subseteq \limfunc{Eig}(M, \lambda_{\min}(M))$ for every~$i = 1, \hdots, n-1$, implying~$\limfunc{span}(v, v_i)^{\bot} \subseteq \limfunc{span}(v_1, \hdots, v_{i-1}, v_{i+1}, \hdots, v_{n-1})$ for every~$i = 1, \hdots, n-1$, which, by a dimension argument using~$v \notin \limfunc{Eig}(M, \lambda_{\min}(M))$, is equivalent to \begin{equation}
\limfunc{span}(v, v_i)^{\bot} = \limfunc{span}(v_1, \hdots, v_{i-1}, v_{i+1}, \hdots, v_{n-1}) \quad \text{ for } i = 1, \hdots, n-1;
\end{equation}
or equivalently 
\begin{equation}
\limfunc{span}(v, v_i) = \limfunc{span}(v_1, \hdots, v_{i-1}, v_{i+1}, \hdots, v_{n-1})^{\bot} \quad \text{ for } i = 1, \hdots, n-1.
\end{equation}
Since~$n \geq 3$, setting~$i = 1$ and~$i = 2$ in the previous display then shows that~$v$ is orthogonal to~$v_1, \hdots, v_{n-1}$, and hence~$\limfunc{span}(v) = \limfunc{Eig}(M, \lambda_{\max}(M))$ would follow. But then we could conclude that~$M= \lambda_{\min}(M)I_n + (\lambda_{\max}(M)-\lambda_{\min}(M))vv'$, a contradiction. Now, let~$x \in \limfunc{Eig}(M, \lambda_{\min}(M))$ be such that~$\|x\| = 1$ and a corresponding~$u \in \limfunc{span}(v,x)^{\bot}$ such that~$\|u\| = 1$ and such that~$u'Mu > \lambda_{\min}(M)$. Let~$b_m \neq 0$ be a sequence that converges to~$0$ and such that~$b_m \neq -v'x$ holds for every~$m \in \N$. Then, we define~$v_m := x + b_m v ~\bot~ u$ and set~$L_{m, \bot} := (u, v_m, w_1, \hdots, w_{n-d-2})$ (with~$L_{m, \bot} = (u, v_m)$ in case~$d = n-2$), for~$w_1, \hdots, w_{n-d-2}$ linearly independent elements of~$\limfunc{span}(u, v, x)^{\bot}$ (which is possible as~$d \geq 1$). As~$v_m \neq 0$ follows from~$b_m \neq -v'x$, the matrix~$L_{m, \bot}$ has full column rank~$n-d \geq 2$ for every~$m \in \N$. Now, for every~$m \in \N$ set~$L_m$ equal to an~$n \times d$ matrix such that~$L_m'L_m = I_d$ and~$\limfunc{span}(L_m)^{\bot} = \limfunc{span}(L_{m, \bot})$. Then
\begin{equation}
\Pi_{\limfunc{span}(L_m)^{\bot}}v = \Pi_{\limfunc{span}(L_{m, \bot})}v = \Pi_{\limfunc{span}((u, v_m))}v = \Pi_{\limfunc{span}((v_m))}v = a_m v_m,
\end{equation}
where~$a_m = (v'x + b_m)/(v_m'v_m) \neq 0$ holds for all~$m$. From~$v_m \neq 0$, we thus obtain~$\Pi_{\limfunc{span}(L_m)^{\bot}}v \neq 0$ for every~$m \in \N$. But~$v_m \to x$ hence shows that
\begin{equation}
a_m^{-2} v'\Pi_{\limfunc{span}(L_m)^{\bot}} M\Pi_{\limfunc{span}(L_m)^{\bot}}v \to \lambda_{\min}(M) \quad \text{ and } \quad a_m^{-2}v'\Pi_{\limfunc{span}(L_m)^{\bot}} v \to 1,
\end{equation}
which implies~\eqref{eqn:eigconv}. Equation~\eqref{eqn:maxconv} follows because~$u \in \limfunc{span}(L_m)^{\bot}$ gives~$\Pi_{\limfunc{span}(L_m)^{\bot}}u = u$, and since~$u$ was chosen such that~$\|u\| = 1$ and~$u'Mu > \lambda_{\min}(M)$. 

2) Obviously,~$C_{L_m}v \neq 0$ follows from~$\Pi_{\limfunc{span}(L_m)^{\bot}}v \neq 0$. Consider first Equation~\eqref{eqn:eigconv2}. Let~$m'$ be an arbitrary subsequence of~$m$. Define~$v_m := C_{L_m} v /\|C_{L_m}v \|$ and~$A_m := C_{L_m} M C_{L_m}'$. Clearly~$\|v_m\|=1$, and~$A_m$ is a norm-bounded sequence because~$C_{L_m} C_{L_m}' = I_{n-d}$. The latter also implies 
\begin{equation}\label{eqn:eigenbound}
\lambda_{1}(M) \leq \lambda_{1}(A_m) \leq \lambda_{n-d}(A_m) \leq \lambda_n(M) \quad \text{ for every } m \in \N.
\end{equation}
Hence, we can choose a subsequence~$m''$ of~$m'$, say, along which~$v_m$ and~$A_m$ converge to~$v_*$ and~$A$, say, respectively. Note that~$\|v_*\| = 1$. Next, we use~$C_{L_m}'C_{L_m} = \Pi_{\limfunc{span}(L_m)^{\bot}}$ to rewrite~$$v' \Pi_{\limfunc{span}(L_m)^{\bot}} M  \Pi_{\limfunc{span}(L_m)^{\bot}} v / v' \Pi_{\limfunc{span}(L_m)^{\bot}}v = v_m' C_{L_m}M C_{L_m}' v_m = v_m' A_m v_m,$$ and use Equation~\eqref{eqn:eigconv} to conclude that along~$m''$ we have~$v_m' A_m v_m \to v_*'Av_* = \lambda_{\min}(M)$. From Equation~\eqref{eqn:eigenbound} we obtain~$\lambda_{\min}(M) = \lambda_{\min}(A)$, hence~$$v_* \in \limfunc{Eig}(A, \lambda_{1}(A)) = \limfunc{Eig}(F(A), \lambda_1(F(A))),$$ where the equality is obtained from~\eqref{eqn:eigenspaceeq}. Finally, we observe that along~$m''$ we have (using continuity of~$F$) that~$B(L_m) = F(A_m) \to F(A)$, from which 
\begin{equation}
v' C_{L_m}' B(L_m)C_{L_m}v /\|C_{L_m}v\|^2 = v_m' F(A_m) v_m \to v_*' F(A) v_* = \lambda_1(F(A)),
\end{equation}
and~$\lambda_{1}(B(L_m)) \to \lambda_{1}(F(A))$ follows (along~$m''$). Hence, we have shown that the statement in Equation~\eqref{eqn:eigconv2} holds along the subsequence~$m''$ of~$m'$. But~$m'$ was arbitrary. Therefore, we are done. 

For~\eqref{eqn:maxconv2} we argue by contradiction. Note first that the limit inferior in~\eqref{eqn:maxconv2} can not be infinite, because~$B(L_m) = F(C_{L_m} M C_{L_m}')$, and the continuity property of~$F$ together with boundedness of~$C_{L_m} M C_{L_m}'$. Now, assuming~\eqref{eqn:maxconv2} were false, we could choose a subsequence~$m'$ of~$m$ such that~$\lambda_{n-k}(B(L_{m'})) - \lambda_1(B(L_{m'})) \to 0$. Choose a subsequence~$m''$ of~$m'$ along which~$v_m$ just defined above,~$u_m := C_{L_m} u / \| C_{L_m} u\|$ (note that~$C_{L_m} u \neq 0$ follows from~$\Pi_{\limfunc{span}(L_m)^{\bot}}u \neq 0$) and~$A_m := C_{L_m}M C_{L_m}'$ converge to~$v_*$,~$u_*$ and~$A$, respectively (where~$v_*$ and~$A$ might differ from the limits in the preceding paragraph where we established Equation~\eqref{eqn:eigconv2}). Note also that~$\|v_*\| = \|u_*\| = 1$. Recall that~$B(L_m) = F(A_m)$, note that
\begin{equation}
\lambda_{n-k}(F(A_{m})) \geq u'C_{L_{m}}' F(A_{m}) C_{L_{m}} u / \| C_{L_{m}} u\|^2 = u_m'F(A_m)u_m \geq \lambda_{1}(F(A_{m})),
\end{equation}
and that, using~$\lambda_{n-k}(F(A_{m'})) - \lambda_1(F(A_{m'})) \to 0$ together with continuity of~$F$ at~$A$, the upper and lower bound in the previous display converge along~$m''$ to~$\lambda_1(F(A))$. It follows that~$u_*'F(A)u_*' = \lambda_1(F(A))$, and hence~$u_* \in \limfunc{Eig}(F(A), \lambda_1(F(A))) = \limfunc{Eig}(A, \lambda_1(A))$, the equality following from Equation~\eqref{eqn:eigenspaceeq}. But from Equation~\eqref{eqn:maxconv} we conclude that~$\lambda_{\min}(M) < c = u_m'A_m u_m = u_*'A u_* = \lambda_1(A)$ holds. To arrive at a contradiction it suffices to show that~$\lambda_{\min}(M) = \lambda_{\min}(A)$. But (similar as argued above in the proof of \eqref{eqn:eigconv2}) this follows from Equation~\eqref{eqn:eigconv}, showing that~$v_m'A_m v_m \to v_*'Av_* = \lambda_{\min}(M)$ along~$m''$, together with Equation~\eqref{eqn:eigenbound}.
\end{proof}

\begin{proof}[Proof of Proposition~\ref{prop:mart}:]
We start with (1.): Let~$\alpha \in (0, 1)$. Let~$X_m$ be a sequence of~$n \times k$-dimensional orthonormal matrices converging to some~$Z \in \R^{n \times k}$ orthonormal, such that~$e \notin \limfunc{span}(X_m)$ holds for every~$m \in \N$, such that
\begin{equation}\label{eqn:limtbxmeEV}
T_{B(X_m), C_{X_m}}(e)- \lambda_1(B(X_m)) = e'C_{X_m}'B(X_m)C_{X_m}e/\|C_{X_m}e\|^2 - \lambda_1(B(X_m)) \to 0,
\end{equation}
and such that~$\liminf_{m \to \infty} \lambda_{n-k}(B(X_m)) - \lambda_1(B(X_m)) = \delta > 0$, and where~$\delta$ is a real number. Such a sequence exists as a consequence of Part~2 of Lemma~\ref{lem:exseq} (applied with~$d = k$ and~$v = e$). Without loss of generality, passing to a subsequence if necessary, we assume that~$\lambda_{n-k}(B(X_m)) - \lambda_1(B(X_m)) > 0$ holds for every~$m \in \N$. Denote by~$\kappa_m$ the critical value~$\kappa(\alpha)$ corresponding to~$\Phi_{B(X_m), C_{X_m}, \kappa(\alpha)}$, cf.~Lemma~\ref{lem:sizecontrol}, and recall from that lemma that~$\lambda_1(B(X_m)) < \kappa_m < \lambda_{n-k}(B(X_m))$ then holds as~$\alpha \in (0, 1)$. 
Passing to a subsequence if necessary, we can assume that~$C_{X_m}$ converges to~$D_Z$, say, an~$(n-k) \times n$ matrix the rows of which form an orthonormal basis of~$\limfunc{span}(Z)^{\bot}$. Recall the continuity property of~$F$ and that~$B(X_m) = F(C_{X_m} M C_{X_m}')$. It follows that~$B(X_m)$,~$\lambda_1(B(X_m))$,~$\lambda_{n-k}(B(X_m))$ converge to~$H:= F(D_Z M D_Z')$,~$b:=\lambda_1(H)$ and~$c:= \lambda_{n-k}(H)$, respectively, with~$c - b \geq \delta > 0$. Passing to another subsequence, if necessary, we can additionally achieve that~$\kappa_m \to \kappa^*$, say. Obviously,~$b \leq \kappa^* \leq c$ holds. We now argue that~$b < \kappa^*$ must hold: By the definition of~$\kappa_m$~$$\alpha = P_{0, 1, 0}^{X_m}(\Phi_{B(X_m), C_{X_m}, \kappa(\alpha)}) = P_{0, 1, 0}^{X_m}(\Phi_{B(X_m), C_{X_m}, \kappa_m}) = P^{X_m}_{0, 1, 0}(\{y \in \R^n: T_{B(X_m)}(y) > \kappa_m\}).$$ Denoting by~$G_m$ the~cdf.~of the image measure~$P^{X_m}_{0, 1, 0} \circ T_{B(X_m),C_{X_m}}$ this implies~$1-\alpha = G_m(\kappa_m)$. From Lemma~B.4 of \cite{PP17} we obtain that the support of~$P^{X_m}_{0, 1, 0} \circ T_{B(X_m),C_{X_m}}$ coincides with~$[\lambda_1(B(X_m)), \lambda_{n-k}(B(X_m))]$, that~$G_m$ is a continuous function, and that~$G_m$ is strictly increasing on~$[\lambda_1(B(X_m)), \lambda_{n-k}(B(X_m))]$. Hence, from~$1-\alpha \in (0, 1)$, it follows that~$G_m^{-1}(1-\alpha) = \kappa_m$, where~$G_m^{-1}$ denotes the quantile function corresponding to~$G_m$. It is easy to see that~$G_m$ converges in distribution to the cdf.~$G$, say, of~$P_{0, 1, 0}^{Z} \circ T_{H, D_Z}$, where the function~$T_{H, D_Z}: \R^n \to \R$ is defined as
\begin{equation}
T_{H, D_Z}(y) = \begin{cases}
y'D_Z' H D_Z y/\|D_Zy\|^2 & \text{ if } y \notin \limfunc{span}(Z), \\
\lambda_1(H) & \text{ else}.
\end{cases}
\end{equation}
Again, Lemma~B.4 of \cite{PP17} (with ``$B = H$ and~$C_X = D_Z$'') shows that the support of~$G$ is~$[b, c]$, that~$G$ is continuous (recall that~$c-b \geq \delta > 0$), and that~$G$ is strictly increasing on~$[b,c]$. This implies that the quantile function~$G^{-1}$ corresponding to~$G$ is continuous on~$(0, 1)$, and that~$G^{-1}(1-\alpha) > b$. Using the convergence in distribution pointed out above, we conclude that the quantiles~$\kappa_m = G_m^{-1}(1-\alpha) \to G^{-1}(1-\alpha) = \kappa^* > b$. Using Equation~\eqref{eqn:limtbxmeEV} can now conclude that there exists an~$m_* \in \N$ such that~$X_{m_*} =: X_*$ is of full column rank, such that~$e \notin \limfunc{span}(X_*)$, such that~$\lambda_1(B(X_*)) < \lambda_{n-k}(B(X_*))$, and such that~$T_{B(X_*), C_{X_*}}(e) < \kappa_{m_*}$ (with~$\kappa_{m^*}$ the critical value~$\kappa(\alpha)$ corresponding to~$\Phi_{B(X_*), C_{X_m}, \kappa(\alpha)}$ and~$\alpha \in (0, 1)$). Theorem~\ref{lem:zptphiPP17} establishes~$X_* \in \mathscr{X}(\alpha; B)$.

\smallskip

We now prove (2.): Recall that~$X_*$ has full column rank and~$e\notin\limfunc{span}(X_*)$. We conclude that both statements (i)~$X$ is of full column rank and (ii)~$e \notin \limfunc{span}(X)$ hold for every~$X$ in an open set~$\mathscr{N}$, say, containing~$X_*$. We now claim that~$$T_{B(X), C_X}(e) < \overline{\kappa}(X); \text{ with } \overline{\kappa}(X) \in (\lambda_1(B(X)), \lambda_{n-k}(B(X))) \text{ s.t.~} P_{0, 1, 0}(\Phi_{B(X), C_X, \overline{\kappa}(X)} ) = \alpha,$$ holds for every~$X$ in an open set~$\mathscr{O} \subseteq \mathscr{N}$ containing~$X_{*}$ (that~$X_*$ satisfies the display was just shown above). Arguing as above, this claim and Theorem~\ref{lem:zptphiPP17} (together with Lemma~\ref{lem:sizecontrol}) would imply~$\mathscr{O} \subseteq \mathscr{X}(\alpha; B)$, and we were done. To prove the claim, it suffices to verify that~$T_{B(X), C_X}(e)$ and~$\overline{\kappa}(X)$ as in the previous display are (well defined) continuous functions of~$X$ on a neighborhood of~$X_*$. First, in order to ensure via Lemma~\ref{lem:sizecontrol} that a~$\overline{\kappa}(X)$ as in the previous display uniquely exists on a neighborhood of~$X_*$, we show that~$\lambda_1(B(X)) < \lambda_{n-k}(B(X))$ holds on an open subset of~$\mathscr{N}$ containing~$X_*$. Recalling that~$\lambda_1(B(X_*)) < \lambda_{n-k}(B(X_*))$, and noting that the map~$y \mapsto C_{X_*}y$ is a surjection of~$\R^n \backslash \limfunc{span}(X_*)$ to~$\R^{n-k} \backslash \{0\}$, we conclude that there exist two vectors~$y_1$ and~$y_2$ in~$\R^n \backslash \limfunc{span}(X_*)$, and such that~$$\lambda_1(B(X_*)) = T_{B(X_*), C_{X_*}}(y_1) < T_{B(X_*), C_{X_*}}(y_2) =\lambda_{n-k}(B(X_*)).$$ holds. From the additional continuity property in~(2.)~it follows that~$y_1, y_2 \notin \limfunc{span}(X)$ and~$T_{B(X), C_{X}}(y_1) < T_{B(X), C_{X}}(y_2)$ hold on an open set~$\mathscr{O}_1 \ni X_*$, say, such that~$\mathscr{O}_1 \subseteq \mathscr{N}$, from which it follows that for every~$X \in \mathscr{O}_1$ we have~$\lambda_1(B(X)) < \lambda_{n-k}(B(X))$. From~$\mathscr{O}_1 \subseteq \mathscr{N}$ we conclude from Lemma~\ref{lem:sizecontrol} that a~$\overline{\kappa}(X)$ satisfying the property to the right in penultimate  display uniquely exists for every~$X \in \mathscr{O}_1$. Since~$X \mapsto T_{B(X), C_X}(e)$ is continuous on~$\mathscr{O}_1 \subseteq \mathscr{N}$ by assumption, it remains to verify that~$X \mapsto \overline{\kappa}(X)$ is continuous on~$\mathscr{O}_1$. Lemma~B.4 of~\cite{PP17} and the definition of~$\overline{\kappa}(X)$ show that for~$X \in \mathscr{O}_1$ we have~$\overline{\kappa}(X) = F_X^{-1}(1-\alpha)$, where~$F_X$ denotes the cdf.~of the image measure~$P_{0, 1, 0} \circ T_{B(X), C_X}$. It is easy to see (using the additional continuity condition in~(2.)) that the map~$X \mapsto F_X$ is continuous on~$\mathscr{O}_1$ (equipping the co-domain with the topology of weak convergence). Furthermore, for every~$X \in \mathscr{O}_1$ it holds (via Lemma~B.4 in \cite{PP17}) that~$P_{0, 1, 0} \circ T_{B(X), C_X}$ has support~$[\lambda_1(B(X)), \lambda_{n-k}(B(X))]$ (which is non-degenerate), that the cdf.~$F_X$ is continuous, and strictly increasing on~$[\lambda_1(B(X)), \lambda_{n-k}(B(X))]$. Hence, for every~$X \in \mathscr{O}_1$ the quantile function~$F_X^{-1}$ is continuous at~$1-\alpha \in (0, 1)$. Continuity of~$X \mapsto \overline{\kappa}(X) = F_X^{-1}(1-\alpha)$ on~$\mathscr{O}_1$ follows.
\end{proof}

\begin{proof}[Proof for the claim made in Remark~\ref{rem:summary}:]
We verify that for~$B(X) = -(C_X \Sigma(\overline{\rho})C_X')^{-1}$, $\overline{\rho} \in (0, a)$, and every~$z \in \R^n$ the function~$X \mapsto T_{B(X), C_X}(z)$ is continuous at every~$X \in \R^{n \times k}$ of full column rank such that~$z \notin \limfunc{span}(X)$. Fix~$z \in \R^n$. Let~$X$ be of full column rank such that~$z \notin \limfunc{span}(X)$, and let~$X_m$ be a sequence converging to~$X$. Eventually,~$X_m$ is of full column rank and satisfies~$z \notin \limfunc{span}(X_m)$, hence we may assume that this is the case for the whole sequence. We need to show that as~$m \to \infty$ we have~$T_{B(X_m), C_{X_m}}(z) \to T_{B(X), C_X}(z)$, or equivalently that
\begin{equation}
\frac{z'C_{X_m}' (C_{X_m} \Sigma(\overline{\rho})C_{X_m}')^{-1} C_{X_m} z}{z' \Pi_{\limfunc{span}(X_m)^{\bot}} z} \to \frac{z'C_{X}' (C_{X} \Sigma(\overline{\rho})C_{X}')^{-1} C_{X} z}{z' \Pi_{\limfunc{span}(X)^{\bot}} z}.
\end{equation}
Since~$X$ is of full column rank~$z' \Pi_{\limfunc{span}(X_m)^{\bot}} z \to z' \Pi_{\limfunc{span}(X)^{\bot}} z \neq 0$ obviously holds. For the numerators, let~$m'$ be an arbitrary  subsequence of~$m$, and choose~$m''$ a subsequence of~$m'$ such that along~$m''$ the sequence~$C_{X_m}$ converges to~$D$, say. Note that~$D$ is necessarily orthonormal and~$\limfunc{span}(D) = \limfunc{span}(X)^{\bot}$. Hence, along~$m''$, noting that~$\Sigma(\overline{\rho})$ is positive definite by assumption, we have~$z'C_{X_m}' (C_{X_m} \Sigma(\overline{\rho})C_{X_m}')^{-1} C_{X_m} z \to z'D' (D \Sigma(\overline{\rho})D')^{-1} D z$. Since~$D = UC_X$ holds for an~$(n-k)\times(n-k)$ orthonormal matrix~$U$, say, it follows that 
\begin{equation}
z'D' (D \Sigma(\overline{\rho})D')^{-1} D z = z'C_X' U' (UC_X \Sigma(\overline{\rho})C_X'U')^{-1} UC_X z = z'C_{X}' (C_{X} \Sigma(\overline{\rho})C_{X}')^{-1} C_{X} z.
\end{equation}
Since the subsequence~$m'$ was arbitrary, we are done.
\end{proof}

\section{Proofs for results in Section~\ref{sec:avoid}}\label{app:avoid}

\begin{proof}[Proof of Theorem~\ref{thm:artreg}:]
Denote by~$\bar{P}_{(\beta, \gamma), \sigma, \rho}$ the distribution induced by~\eqref{linmod}, but where~$X$ is replaced by~$\bar{X} = (X, e)$ (a matrix with column rank~$k+1<n$), and where~$\gamma$ is the regression coefficient corresponding to~$e$. Note also that for every~$\beta \in \R^k$, every~$\sigma \in (0, \infty)$ and every~$\rho \in [0, a)$ the measure~$\bar{P}_{(\beta, 0), \sigma, \rho}$ coincides with~$P_{\beta, \sigma, \rho}$. An application of Corollary~2.22 in \cite{PP17} (recall that~$\bar{\kappa}(\alpha) \in (\lambda_1(\bar{B}) < \lambda_{n-k-1}(\bar{B}))$ from the discussion preceding Equation~\eqref{eqn:corsize}, and acting as if~$\bar{X}$ was the underlying design matrix) one then immediately obtains that for every~$\beta \in \R^k$, every~$\sigma \in (0, \infty)$ and every~$\gamma \in \R$ it holds that
\begin{equation}
0 < \lim_{\rho \to a} \bar{P}_{(\beta, \gamma), \sigma, \rho}(\bar{\Phi}_{\bar{B}, \bar{\kappa}(\alpha)}) = \mathrm{Pr}(\bar{T}_{\bar{B}}(\Lambda \mathbf{G}) > \bar{\kappa}(\alpha)) < 1.
\end{equation}
Setting~$\gamma = 0$ then delivers the claim.
\end{proof}

\begin{proof}[Proof of Proposition~\ref{prop:excrit}:]
We proceed in~3 steps:

\smallskip

1) By a simple~$G_X$-invariance argument (recall~A.1 and that~$T_{C_Xee'C_X'}$ is~$G_X$-invariant) it suffices to verify that  for every~$\alpha \in (0, 1)$ and every~$\varepsilon \in (0, \alpha)$ there exists a~$c(\alpha, \varepsilon) \in (0, \kappa(\varepsilon)]$ such that 
\begin{equation}\label{eqn:defc2}
E_{0, 1, 0}\left[ \min \left(\varphi_{\alpha - \varepsilon} + \mathbf{1}_{\Phi_{C_Xee'C_X', c(\alpha, \varepsilon)}},~ 1\right) \right] =  \alpha,
\end{equation}
and such that for every~$c' \in (0, c(\alpha, \varepsilon))$ it holds that the supremum in the previous display is greater than~$\alpha$. 

2) We claim that the non-increasing function~$g: \R \to \R$ defined via
$$c \mapsto E_{0, 1, 0}\left[\min \left(\varphi_{\alpha - \varepsilon} + \mathbf{1}_{\Phi_{C_Xee'C_X', c}},~ 1\right)\right]$$ is continuous. To verify this claim let~$c \in \R$, and let~$c_m \to c$ be a real sequence. By the Dominated Convergence Theorem, to show that~$g(c_m) \to g(c)$ holds, it is enough to verify 
\begin{equation}
\lim_{m \to \infty} \left[\min \left(\varphi_{\alpha - \varepsilon}(y) + \mathbf{1}_{\Phi_{C_Xee'C_X', c_m}}(y),~ 1\right)\right] = \left[\min \left(\varphi_{\alpha - \varepsilon}(y) + \mathbf{1}_{\Phi_{C_Xee'C_X', c}}(y),~ 1\right)\right] 
\end{equation}
for~$P_{0, 1, 0}$-almost every~$y \in \R^n$. It suffices to verify that
\begin{equation}
\lim_{m \to \infty} \mathbf{1}_{\Phi_{C_Xee'C_X', c_m}}(y) = \mathbf{1}_{\Phi_{C_Xee'C_X', c}}(y)
\end{equation}
holds for~$P_{0, 1, 0}$-almost every~$y \in \R^n$. The statement in the previous display holds for every~$y$ such that~$T_{C_X ee'C_X'}(y) \neq c$. The claim now follows from~$P_{0, 1, 0}(\{y \in \R^n: T_{C_X ee'C_X'}(y) = c\}) = 0$, which can be obtained from Part~1 of Lemma~B.4 in \cite{PP17} upon noting that~$\lambda_1(C_X ee'C_X') = 0$ (recall that~$k<n-1$) and that~$0 < \|C_Xe\|^2 = \lambda_{n-k}(C_X ee'C_X')$ (the inequality following from~$e \notin \limfunc{span}(X)$).

3) Next, note that~$\alpha - \varepsilon \leq g \leq 1$ (using~A.2 for the lower bound). Observe that~$g(0) = 1$ follows from~$1 \geq g(0) \geq P_{0, 1, 0}(\Phi_{C_X ee'C_X',0}) = 1$, the last equality following from Part~1 of Lemma~B.4 in \cite{PP17}. Observe also that~$g(\|C_X e\|^2) = \alpha - \varepsilon$ follows from~$\alpha - \varepsilon \leq g(\|C_X e\|^2) \leq \alpha - \varepsilon + P_{0, 1, 0}(\Phi_{C_X ee'C_X', \|C_X e\|^2}) = \alpha - \varepsilon$, the last equality following again from  Part~1 of Lemma~B.4 in \cite{PP17}. 
%
%
From these two observations, monotonicity of~$g$, and the continuity of~$g$ it follows that~$\{c \in \R: g(c) = \alpha\}$ is a closed interval contained in~$(0, \|C_X e\|^2)$. Define~$c(\alpha, \varepsilon)$ as the lower endpoint of this closed interval. Equation~\eqref{eqn:defc2} and thus Equation~\eqref{eqn:defc} follows. Furthermore, since~$c(\alpha, \varepsilon)$ was defined as the lower endpoint, monotonicity of~$g$ implies that every~$c' \in (0, c(\alpha, \varepsilon))$ must satisfy~$g(c') > g(c) = \alpha$. To finally show that~$c(\alpha, \varepsilon) \leq \kappa(\varepsilon)$ holds, suppose the opposite, from which it follows from what was already shown that~$g(\kappa(\varepsilon)) > \alpha$, which is obviously false (cf.~the discussion surrounding~\eqref{eqn:couldusethis}). Note also that~$0 < \kappa(\varepsilon) < \|C_X e\|^2$ follows from Lemma~\ref{lem:sizecontrol}.
\end{proof}

\begin{proof}[Proof of Theorem~\ref{thm:enhthm}:]
1.) Let~$\varepsilon \in (0, \alpha)$. Obviously
\begin{equation}
\varphi^*_{\alpha, \varepsilon} \geq \mathbf{1}_{\Phi_{C_Xee'C_X', c(\alpha, \varepsilon)}},
\end{equation}
which shows that for every~$\beta \in \R^k$, every~$\sigma \in (0, \infty)$ and every~$\rho \in [0,a)$ we have%
\begin{equation}
E_{\beta, \sigma, \rho}(\varphi^*_{\alpha, \varepsilon}) \geq P_{\beta, \sigma, \rho}(\Phi_{C_Xee'C_X', c(\alpha, \varepsilon)}).
\end{equation}
From Proposition~\ref{prop:excrit} we know that~$0 = \lambda_1(C_X ee'C_X') < c(\alpha, \varepsilon) < \lambda_{n-k}(C_X ee'C_X') = \|C_Xe\|^2$. We can therefore use Lemma~\ref{lem:sizecontrol} (with~$B = C_X ee'C_X'$) to conclude that~$c(\alpha, \varepsilon) = \kappa(\alpha^*)$ for some~$\alpha^* \in (0, 1)$, and apply Theorem~\ref{thm:ee} to conclude that for every~$\beta \in \R^k$ and every~$\sigma \in (0, \infty)$ we have~$\lim_{\rho \to a} P_{\beta, \sigma, \rho}(\Phi_{C_X ee'C_X', c(\alpha, \varepsilon)}) = 1$, which together with the lower bound in the previous display proves the claim.

2.) Using~$G_X$-invariance of~$\varphi^*_{\alpha, \varepsilon}$ (for every~$\varepsilon \in (0, \alpha)$) and of~$\varphi_{\alpha}$, together with~$\|\Sigma(\rho)\| > 0$ for every~$\rho \in [0, a)$, it suffices to verify  that
\begin{equation}
\lim_{\varepsilon \to 0^+} \sup_{\rho \in A} |E_{0, \|\Sigma(\rho)\|^{-1/2}, \rho}(\varphi^*_{\alpha, \varepsilon}) - E_{0, \|\Sigma(\rho)\|^{-1/2}, \rho}(\varphi_{\alpha})| = 0.
\end{equation}
Let~$\varepsilon_m \to 0$ be a sequence in~$(0, \alpha)$ and let~$\rho_m$ be a sequence in~$A$. For convenience, set~$\sigma_m := \|\Sigma(\rho_m)\|^{-1/2}$. We verify that 
\begin{equation}\label{eqn:seqcont2}
|E_{0, \sigma_m, \rho_m}(\varphi^*_{\alpha, \varepsilon_m}) - E_{0, \sigma_m, \rho_m}(\varphi_{\alpha})| = |E_{0, \sigma_m, \rho_m}(\varphi^*_{\alpha, \varepsilon_m} - \varphi_{\alpha})| \to 0.
\end{equation}
Let~$m'$ be an arbitrary subsequence of~$m$. By compactness of the unit sphere in~$\R^{n \times n}$, we can choose a subsequence~$m''$ of~$m'$ along which~$\|\Sigma(\rho_m)\|^{-1} \Sigma(\rho_m)$ converges to a symmetric matrix~$\Gamma$, say, which due to the additional assumption on the set~$A$ is positive definite. It follows from Scheff\'e's lemma that along~$m''$ the sequence~$P_{0,\sigma_m,\rho_m}$ (i.e., the Gaussian probability measure with mean~$0$ and covariance matrix~$\|\Sigma(\rho_m)\|^{-1}\Sigma(\rho_m)$) converges in total-variation-distance to~$Q$, a Gaussian probability measure with mean~$0$ and covariance matrix~$\Gamma$. Obviously~$|E_{0, \sigma_m, \rho_m}(\varphi^*_{\alpha, \varepsilon_m} - \varphi_{\alpha})| \leq 2 E_{0, \sigma_m, \rho_m}(.5|\varphi^*_{\alpha, \varepsilon_m} - \varphi_{\alpha}|)$. By, e.g., Lemma~2.3 in \cite{strasser1985mathematical} and since~$.5|\varphi^*_{\alpha, \varepsilon_m} - \varphi_{\alpha}|$ is a sequence of tests, it follows from the total variation convergence established above that along~$m''$ we have
\begin{equation}
|E_{0, \sigma_m, \rho_m}(|\varphi^*_{\alpha, \varepsilon_m} - \varphi_{\alpha}|) - E_{Q}(|\varphi^*_{\alpha, \varepsilon_m} - \varphi_{\alpha}|)| \to 0,
\end{equation}
where~$E_{Q}$ denotes expectation w.r.t.~$Q$. We now claim that~
\begin{equation}\label{eqn:QconvEnh}
E_{Q}(|\varphi^*_{\alpha, \varepsilon_m} - \varphi_{\alpha}|) \to 0.
\end{equation}
This claim, if true, then implies Equation~\eqref{eqn:seqcont2} as the subsequence~$m'$ we started with was arbitrary. We first show that the sequence in the previous display converges to~$0$, when the expectation is taken w.r.t.~$P_{0, 1, 0}$ instead of~$Q$. To this end write
\begin{equation}\label{eqn:testalt}
\varphi^*_{\alpha, \varepsilon_m} - \varphi_{\alpha} = [\varphi_{\alpha - \varepsilon_m} - \varphi_{\alpha}] + (1-\varphi_{\alpha - \varepsilon_m}(y))\mathbf{1}_{\Phi_{C_Xee'C_X', c(\alpha, \varepsilon_m)}}.
\end{equation}
From~A.3 and the Dominated Convergence Theorem we obtain~$E_{0, 1, 0}[|\varphi_{\alpha - \varepsilon_m} - \varphi_{\alpha}|] \to 0$. It remains to show that~$E_{0, 1, 0}(\psi_m) \to 0$ for~$\psi_m := (1-\varphi_{\alpha - \varepsilon_m}(y))\mathbf{1}_{\Phi_{C_Xee'C_X', c(\alpha, \varepsilon_m)}} \geq 0$. By construction and~A.2, however, we have~$E_{0, 1, 0}(\varphi^*_{\alpha, \varepsilon_m}) = \alpha = E_{0, 1, 0}(\varphi_{\alpha})$. Therefore, the preceding display shows that~$-E_{0, 1, 0}[\varphi_{\alpha - \varepsilon_m} - \varphi_{\alpha}] = E_{0,1,0}(\psi_m)$. The statement hence follows from~$E_{0, 1, 0}[\varphi_{\alpha - \varepsilon_m} - \varphi_{\alpha}] \to 0$. Now,  suppose~\eqref{eqn:QconvEnh} were false. Then, there would exist a subsequence~$m^{\star}$ of $m$ along which the sequence in~\eqref{eqn:QconvEnh} converges to~$b > 0$, say. Since~$E_{0, 1, 0}(|\varphi^*_{\alpha, \varepsilon_{m^{\star}}} - \varphi_{\alpha}|) \to 0$, there exists a subsequence~$m^{\star \star}$ of~$m^{\star}$ and a set~$N$ such that~$P_{0, 1, 0}(N) = 0$, and such that for every~$y \in \R^n \backslash N$ it holds that~$|\varphi^*_{\alpha, \varepsilon_{m^{\star \star}}}(y) - \varphi_{\alpha}(y)| \to 0$~(cf., e.g., Theorem 3.12 in \cite{rudin}). From positive-definiteness of~$\Gamma$ it follows, however, that~$Q(N) = 0$, and (by the Dominated Convergence Theorem) that~$0 = \lim_{m^{\star \star} \to \infty }E_Q(|\varphi^*_{\alpha, \varepsilon_{m^{\star \star}}} - \varphi_{\alpha}|) = b$, a contradiction.

\end{proof}

\end{appendix}

\end{document}